\documentclass[]{article}
\usepackage{color,graphicx,amssymb,amsmath,amsthm,a4wide}
\usepackage[colorlinks=true]{hyperref}
\newtheorem{lemma}{Lemma}
\newtheorem{theorem}{Theorem}
\newtheorem{remark}{Remark}
\newtheorem{definition}{Definition}
\newtheorem{corollary}{Corollary}
\newtheorem{proposition}{Proposition}

\newcommand{\diag}{\text{diag}}

\newcommand{\NN}{\mathbb{N}}
\newcommand{\RR}{\mathbb{R}}
\newcommand{\CC}{\mathbb{C}}
\newcommand{\refeq}[1]{(\ref{#1})}

\begin{document}

\title{\itshape Quantum Gate Generation by $T$-Sampling Stabilization}

\author{
Hector Bessa  Silveira\thanks{Department of Automation and Systems (DAS),
Federal University of Santa Catarina (UFSC), Florian\'opolis, Brazil}
\and
Paulo Sergio  Pereira da Silva\thanks{Escola Polit\'{e}cnica -- PTC,
University of S\~ao Paulo (USP), S\~ao
Paulo, Brazil.}
\and
 Pierre  Rouchon\thanks{Centre Automatique et Syst\`emes, Mines ParisTech, Paris, France}
 }

\maketitle

\begin{abstract}
This paper considers right-invariant and controllable driftless
quantum systems with state $X(t)$ evolving  on the unitary group $\mbox{U}(n)$ and $m$
inputs $u=(u_1, \dots, u_m)$. The $T$-sampling stabilization problem is introduced and solved: given any initial condition $X_0$ and any goal state $X_{goal}$,
find a control law $u=u(X,t)$ such that $\lim_{j \to \infty} X(jT) = X_{goal}$ for the closed-loop system.
The purpose is to generate arbitrary quantum gates corresponding to $X_{goal}$. This is achieved by the tracking of $T$-periodic reference trajectories $(\overline{X}_{\mathbf{a}}(t), \overline{u}_{\mathbf{a}}(t))$ of the quantum system that pass by $X_{goal}$ using the framework of Coron's Return Method. The $T$-periodic reference trajectories $\overline{X}_{\mathbf{a}}(t)$ are generated by applying controls $\overline{u}_{\mathbf{a}}(t)$ that are a sum of a finite number $M$ of harmonics of $\sin(2\pi t / T)$, whose amplitudes are parameterized by a vector $\mathbf{a}$. The main result establishes that, for $M$ big enough, $X(jT)$ exponentially converges towards $X_{goal}$ for almost all fixed $\mathbf{a}$, with explicit and completely constructive control laws.
This paper also establishes a stochastic version of this deterministic control
law. The key idea is to randomly choose a different parameter vector of control
amplitudes $\mathbf{a}=\mathbf{a}_j$ at each $t=jT$, and keeping it fixed for $t
\in [jT, (j+1)T)$. It is shown in the paper that $X(jT)$ exponentially converges
towards $X_{goal}$ almost surely. Simulation results have indicated that the
convergence speed of $X(jT)$ may be significantly improved with such stochastic
technique. This is illustrated in the generation of the C--NOT quantum logic
gate on $\mbox{U}(4)$.

\bigskip


\end{abstract}

\section{Introduction}\label{pmpp}

Consider a nonlinear driftless system of the form
\begin{equation}\label{mainsystem}
    \dot{X}(t) = \sum_{k=1}^m u_k(t) g_k(X(t)), \quad X(0)=X_0,
\end{equation}
where $X$ is the state and $u=(u_1, \dots, u_m)$ is the control, and assume that $X_{goal}$ is a given goal state.
The following problem is introduced in this paper:
\begin{definition} (\emph{$T$-sampling stabilization problem}) Given an initial condition $X_0$,
find a control law $u=u(X,t)$ such that $\lim_{j \to \infty} X(jT) = X_{goal}$ for the closed-loop system.
\end{definition}
The reason for considering such relaxed version
of the stabilization problem relies on the known obstructions that arises in the context of driftless systems \cite{Bro83,Pom92}. However, it is possible to stabilize such
systems by means of smooth time-varying feedbacks that are time-periodic \cite{Cor92, Pom92}, but  exponential convergence is impossible with such  smooth control law.  For the less restrictive $T$-sampling stabilization problem,  this paper proves that exponential convergence can be obtained.

In this work one solves the $T$-stabilization problem for controllable right-invariant driftless quantum models evolving on $\mbox{U}(n)$. The purpose is to generate arbitrary quantum gates corresponding to $X_{goal}$. For this end, one considers the tracking of $T$-periodic reference trajectories $(\overline{X}_{\mathbf{a}}(t), \overline{u}_{\mathbf{a}}(t))$ of system \eqref{mainsystem} that pass by $X_{goal}$ using the framework of Coron's Return Method  \cite{Cor92} (see also
\cite{Cor99, Cor07}). The $T$-periodic reference trajectories $\overline{X}_{\mathbf{a}}(t)$ are generated by applying controls $\overline{u}_{\mathbf{a}}(t)$ that are a sum of a finite number $M$ of harmonics of $\sin(2\pi t / T)$, whose amplitudes are parameterized by a vector $\mathbf{a}$. The main deterministic result shows that, for $M$ big enough, $X(jT)$ exponentially converges towards $X_{goal}$ for almost all fixed $\mathbf{a}$, with explicit and completely constructive control laws.
A new Lyapunov function $\mathcal{V}$ that is inspired in a homographic function is introduced. The advantage of such $\mathcal{V}$ with respect to the standard fidelity functions, is that $\mathcal{V}$ decreases without singularities and with no nontrivial LaSalle's invariants, which is an advantage when compared
with other previous results in the literature \cite{Pav02,SilPerRou09,
SilPerRou12}.
The key ingredients of the stability proof is the $T$-periodic version of LaSalle's theorem \cite{LaSalle}, Coron's Return Method \cite{Cor92} and the stabilization techniques of \cite{JurQui78}.

This paper also establishes a stochastic version of the deterministic control law described above\footnote{A brief
summary  of this result without proofs was presented in \cite{SilPerRou13}.}. This is achieved by randomly choosing a different amplitudes vector $\mathbf{a}=\mathbf{a}_j$ at each $t=jT$, and keeping it fixed for $t \in [jT, (j+1)T)$. The main stochastic result shows that $X(jT)$ converges exponentially towards $X_{goal}$ almost surely. The proof relies on a well-known stochastic version of LaSalle's theorem \cite{Kushner71}. Simulation results have indicated that the convergence speed of $X(jT)$ may be significantly improved with such stochastic technique.

The deterministic and the stochastic methods described above are both local in nature, since they require that $X_{goal}$ must not have any eigenvalue equal to $-1$.
However, it is shown that the $T$-stabilization problem may be easily solved globally in a two-step procedure. Furthermore, if one admits a global phase change of $X_{goal}$, which is transparent for quantum systems, then the proposed strategy becomes a global exponential solution of the $T$-sampling stabilization problem.

It follows from the principles of quantum mechanics (in the
Copenhagen interpretation) that the (deterministic) Schr\"odinger
equation that governs its dynamics is valid as long as the system
remains isolated from external measurements. More precisely, a
measurement causes a collapse in (reduction of) the state of the
system, and the relation between the state at the moment of a
measurement and the one immediately after it can no longer be
deduced from the Schr\"odinger equation, but only in a probabilistic
manner by the Projection Postulate (see e.g.\ \cite{Sud86}).
Therefore, feedback control techniques cannot be directly applied.
However, it is possible to perform a computer simulation, and then
apply the recorded inputs in the real quantum system as being an
open-loop control. In the case of quantum systems consisting of
$n$-qubits\footnote{The qubit (quantum bit) is the quantum analog of
the usual bit in classical computation theory \cite{NieChu00}.},
the present method  allows one to generate, in an approximate
manner, arbitrary quantum logic gates that operate on $n$-qubits.

The $T$-sampling stabilization control problem here treated for driftless systems of the
form \eqref{mainsystem} evolving on $\mbox{U}(n)$ can be regarded as a generalization of the motion planning
control problem. The latter was considered for systems evolving on $\mbox{SU(n)}$ in \cite{Spi02}
using optimal control theory, solved in an approximate manner in
\cite{LeoKri95}
based on averaging techniques (see also \cite{SahSal01}
for the case of quantum systems)
and treated in \cite{PerRou08b} for a single qubit quantum system that evolves
on $\mbox{SU}(2)$ by means of a flatness approach.
Based on decompositions
of the Lie group $\mbox{SU}(n)$, \cite{Ale08} (see also the references therein)
considers the problem of finding piecewise-constant inputs
that steer the state of quantum systems of the form (\ref{mainsystem}) with a drift term to an arbitrary final state $X_f \in \mbox{SU}(n)$
in some finite instant of time.
It also treats the problem of reaching a given final state in minimum time.

The paper is organized as follows. Section~\ref{mrs} develops the proposed solution for the $T$-sampling stabilization problem of controllable driftless quantum systems evolving on $\mbox{U}(n)$. The deterministic and stochastic control laws are exhibited in Sections~\ref{mrs}~and~\ref{stochaticlawsalgorithm}, respectively. Section~\ref{qma} exhibits the obtained simulation results in the generation of the C--NOT (Controlled-NOT) quantum logic gate for
a quantum system evolving on $\mbox{U}(4)$ ($n=4$). A comparison between the deterministic and the stochastic control strategies is presented. The main deterministic and stochastic results are given in Section~\ref{deterministicstochasticresults}, and their proofs are developed in Section~\ref{mrtd}~and~\ref{proofstochastic}, respectively.
Section~\ref{proofexpconv} shows that the convergence of both methods, deterministic and stochastic, are exponential. The conclusions are
stated in Section~\ref{sConclusions}. Finally, some proofs and intermediate results were deferred to Appendix.

\section{$T$-Sampling Stabilization of Quantum Models}\label{mainresult}

Consider a right-invariant and controllable driftless (homogeneous) quantum
system
of the form
\begin{equation}\label{cqs}
    \dot{X}(t)= -\iota \sum_{k=1}^{m}u_k(t) H_k X(t) =
\sum_{k=1}^{m}u_k(t) S_k X(t), 
\end{equation}
where $X \in \mbox{U}(n)$ is the state, $\iota \in \CC$ is
the imaginary unit, $S_k = -\iota H_k \in \mathfrak{u}(n)$,
$\mathfrak{u}(n)$ is the Lie algebra associated to the special
unitary group $\mbox{U}(n)$, $u_k \in \RR$ are the controls
and $I$ is the identity matrix\footnote{Recall that a $n$-square
complex matrix $X$ belongs to $\mbox{U}(n)$ if and only if $X^\dag
X = I$, and $S$ is in $\mathfrak{u}(n)$ if and
only if $S^\dag = - S$ (skew-Hermitian), where $S^\dag$ is the conjugate transpose of $S$.}.
The approach for solving the $T$-sampling stabilization problem on $\mbox{U}(n)$ for system \eqref{cqs} is described as follows. Given
any initial condition $X(0)=X_0 \in \mbox{U}(n)$ in \eqref{cqs},
any goal state $X_{goal} \in \mbox{U}(n)$ and any $T >
0$, find a smooth $T$-periodic reference trajectory
$\overline{X}$:~$\RR_+ \rightarrow \mbox{U}(n)$ with
$\overline{X}(0)=X_{goal}$, and determine
piecewise-smooth control laws $u_k$:~$\RR_+ \rightarrow
\RR$, $k = 1, \dots, m$, in a manner that
\begin{align}\label{ctez-i}
    \lim_{t \rightarrow \infty} \widetilde X (t) =  \lim_{t \rightarrow \infty} \overline{X}^\dag(t)X(t) = I.
\end{align}
This is illustrated in Figure~\ref{espiral}. In particular, for the sequence of
samples $X(jT)$ one
has
\begin{equation}\label{cgs-i}
    \lim_{j \rightarrow \infty} X(j T) = X_{goal},
\end{equation}
and $T$-sampling stabilization is achieved.

\begin{figure}[!hthb]\label{espiral}
    \centerline{\includegraphics[width=0.4\textwidth]{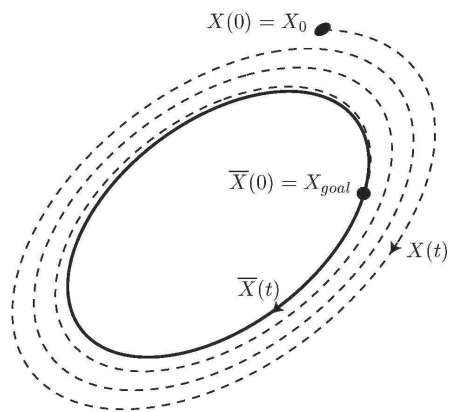}}
    \caption{Proposed approach to solve the $T$-sampling stabilization problem.}\label{pmpcpi}
\end{figure}

One denotes by $\NN$ the set of natural numbers (including
zero), by $\CC^{n \times n}$ the real Banach space of $n \times
n$-square matrices with complex entries endowed with the Frobenius
norm $\| \cdot \|$, and by $\mbox{Tr}(X)$ the trace of $X \in \CC^{n
\times n}$. If $X \in \CC^{n \times n}$ and $\Omega \subset \CC^{n
\times n}$ is nonempty, $d(X,\Omega) \triangleq \inf_{Y \in \Omega}
\| Y - X \|$. The Lie algebra generated by the
$S_k$'s in \eqref{cqs} is denoted as $\mbox{Lie}\{S_1, \dots,
S_m\}$. For simplicity, it will be assumed throughout this paper
that $n \geq 1$, $T > 0$ and the goal state $X_{goal} \in
\mbox{U}(n)$ are fixed. Moreover, unless
otherwise stated, one assumes that the initial condition of
\eqref{cqs} is $X(0)=X_0$, where $X_0 \in \mbox{U}(n)$ is arbitrary but fixed.  The controllability assumption means that the Lie algebra generated
by the $S_k$'s in \eqref{cqs} coincides with $\mathfrak{u}(n)$ \cite{JurSus72}.

\subsection{Deterministic Control Laws}\label{mrs}

Consider system \eqref{cqs}. Take $T > 0$ and set $\omega=2\pi/T$.
Fix an integer $M > 0$ and choose $a_{k,\ell} \in \RR$ for
$k=1,\dots,m$, $\ell=1, \dots, M$. Consider the $T$-periodic
continuous reference controls
\begin{equation}\label{refcon}
  \overline{u}_k(t) = \sum_{\ell=1}^{M} a_{k,\ell} \sin(\ell
\omega t), \quad t \in \RR, \; k=1,\dots,m,
\end{equation}
and the associated reference trajectory $\overline{X}(t) \in
\mbox{U}(n)$, $t \in \RR$, solution of
\begin{equation}\label{rsys}
 \dot{\overline{X}}(t) = \sum_{k=1}^{m} \overline{u}_k(t) S_k
\overline{X}(t), \quad \overline{X}(0)=X_{goal} \in \mbox{U}(n).
\end{equation}
This means that $\overline{X}(t)$ is the solution of \eqref{cqs}
with $u_k=\overline{u}_k$ and initial condition $X_{goal}$ at $t=0$.
Since $\overline{u}_k(T-t)=-\overline{u}_k(t)$ for $t \in
\RR$, one has $\overline{X}(t) = \overline{X}(T-t)$ for $t
\in \RR$, and thus $\overline{X}(jT)=X_{goal}$ for all $j
\geq 0$ (see \cite{Cor07} for details). Therefore, $\overline{X}(t)$
is $T$-periodic.

The tracking error $\widetilde{X}(t)=\overline{X}^{\dag}(t) X(t)$
obeys
\begin{equation}\label{tesys}
  \dot{\widetilde{X}}(t) = \sum_{k=1}^{m} \widetilde{u}_k(t) \widetilde{S}_k(t)
\widetilde{X}(t), \quad \widetilde{X}(0) = X_{goal},
\end{equation}
where $\widetilde{u}_k = u_k - \overline{u}_k$, and
$\widetilde{S}_k(t) = \overline{X}^{\dag}(t) S_k \overline{X}(t) \in
\mathfrak{u}(n)$ depends on $t \in \RR$ and is $T$-periodic.
The goal is to stabilize $\widetilde{X}$ towards the identity $I$.
In order to accomplish this, one will choose a suitable Lyapunov
function $\mathcal{V}(\widetilde{X})$ to measure the distance from
$\widetilde{X}$ to the identity matrix $I$.

Let $\mathcal{W} \subset \mbox{U}(n)$ be the set of $\widetilde{X}
\in \mbox{U}(n)$ which have all the eigenvalues different from
$-1$. Note that
 \begin{equation}
  \label{eW}
\mathcal{W} = \{ W \in \mbox{U}(n) \; | \; \det ( I + W) \neq 0 \}.
 \end{equation}
By
the continuity of the determinant function, it follows that
$\mathcal{W}$ is open in $\mbox{U}(n)$. Denote by $\mathbf{H}$ the
(real) vector space of all Hermitian $n$-square complex matrices and
consider the map\footnote{Given $A, B \in \CC^{n \times n}$ with $B$
invertible and $AB^{-1}=B^{-1}A$ (i.e.\ $A$ and $B^{-1}$ commute),
one defines $A/B \triangleq AB^{-1}=B^{-1}A$. It is easy to see that
$A$ and $B^{-1}$ commute whenever $A$ and $B$ commute.}
\[
 \mathcal{W} \ni  \widetilde{X} \mapsto \Upsilon(\widetilde{X}) = \iota
\frac{\widetilde{X}-I}{\widetilde{X}+I} \in \mathbf{H}.
\]
One has that $\Upsilon$:~$\mathcal{W} \rightarrow \mathbf{H}$ is a
well-defined smooth map on the open submanifold $\mathcal{W}$ of
$\mbox{U}(n)$. Indeed, writing
\[
 \Upsilon(\widetilde{X})=\iota
\dfrac{\widetilde{X}-\widetilde{X}^{\dag}\widetilde{X}}{\widetilde{X}
+\widetilde{X}^{\dag}\widetilde{X}},
\]
one gets
\begin{align*}
 \Upsilon(\widetilde{X})^{\dag} & = -\iota
\dfrac{\widetilde{X}^{\dag}-\widetilde{X}\widetilde{X}^{\dag}}
{\widetilde{X}^{\dag}+\widetilde{X}\widetilde{X}^{\dag}} = -\iota
(I-\widetilde{X})\widetilde{X}^{\dag} (
(I+\widetilde{X})\widetilde{X}^{\dag} )^{-1} =
\Upsilon(\widetilde{X}).
\end{align*}

For $\widetilde{X} \in \mathcal{W}$, the distance to $I$ is measured
by the Frobenius norm\footnote{It follows from \eqref{lyaptan} in
Appendix~\ref{aCompact} that one may always write
$\mathcal{V}(\widetilde{X})=\sum_{j=1}^n \{ \tan(\theta_j/2) \}^2$, where $\exp
(\iota \theta_j), j=1, \ldots, n$, are the eigenvalues of $\widetilde X$.} of
$\Upsilon(\widetilde{X})$:
\begin{equation}\label{lyap}
 \mathcal{V}(\widetilde{X}) \triangleq
\mbox{Tr}\left(\Upsilon(\widetilde{X})
\Upsilon(\widetilde{X})^{\dag} \right) = - \mbox{Tr}\left(
\dfrac{(\widetilde{X}-I)^2}{(\widetilde{X}+I)^2} \right) \geq 0.
\end{equation}
Note that $\mathcal{V}$:~$\mathcal{W} \rightarrow \RR$ is
smooth and non-negative. Moreover, $\mathcal{V}(\widetilde{X})=0$
implies $\widetilde{X}=I$. One will impose that $\dot{\mathcal{V}}
\leq 0$. Using $\mbox{Tr}(AB)=\mbox{Tr}(BA)$, standard computations
yield, for $(\widetilde{X},t) \in \mathcal{W} \times\RR$,
\[
 \dot{\mathcal{V}}(\widetilde{X},t)  \triangleq
\dfrac{d}{dt}\mathcal{V}(\widetilde{X})  = -4 \sum_{k=1}^{m}
\widetilde{u}_k \mbox{Tr}
\left(\widetilde{X}(\widetilde{X}-I)(\widetilde{X}+I)^{-3}\widetilde{S}
_k(t)\right).
\]
Choose any (constant) feedback gains $f_k > 0$. Since $\widetilde{X}(\widetilde{X}-I)(\widetilde{X}+I)^{-3}$ is skew-Hermitian because $ \widetilde{X}$ is unitary, $\mbox{Tr}
\left(\widetilde{X}(\widetilde{X}-I)(\widetilde{X}+I)^{-3}\widetilde{S}
_k(t)\right)$ is real for each $k$  and  the feedbacks
\begin{equation}\label{feedback}
 \widetilde{u}_k(\widetilde{X},t) = f_k \mbox{Tr}
\left(\widetilde{X}(\widetilde{X}-I)(\widetilde{X}+I)^{-3}\widetilde{S}
_k(t)\right)
\end{equation}
produce
\begin{equation}\label{negdef}
 \dot{\mathcal{V}}(\widetilde{X},t) = -4 \sum_{k=1}^{m} f_k \mbox{Tr}^2
\left(\widetilde{X}(\widetilde{X}-I)(\widetilde{X}+I)^{-3}\widetilde{S}
_k(t)\right)  = -4 \sum_{k=1}^{m} \frac{ \widetilde{u}^2_k(\widetilde{X},t)}{f_k} \leq 0,
\end{equation}
$(\widetilde{X},t) \in \mathcal{W} \times\RR$, and hence
ensure that $\mathcal{V}(\widetilde{X}(t))$ is non-increasing.
Note that, for the
initial condition $\widetilde{X}(0)=\widetilde{X}_{goal} \in
\mathcal{W}$, the corresponding solution
$\widetilde{X}(t) \in \mathcal{W}$ is defined for every $t \geq
0$. Indeed, since $\mathcal V$ is decreasing, the trajectory
$\widetilde{X}(t)$ of
the closed-loop system (\ref{tesys},\ref{feedback})
remains in the positively invariant set $K$ given by
 \begin{equation}
  \label{eK}
K=\{\widetilde{X} \in \mathcal{W} \; | \; \mathcal{V}(\widetilde{X})
\leq \mathcal{V}(X_{goal})\},
 \end{equation}
at least in the maximal interval $[0, t_{max})$ where this solution
is defined. Since $K$ is compact (see Appendix \ref{aCompact}), it follows from
well-known results on ordinary differential equations that one
cannot have finite time scape, and thus the solution $\widetilde X(t)$ is
well-defined for all $t \geq 0$.

The controls
$\widetilde{u}_k(t)=\widetilde{u}_k(\widetilde{X}(t),t)$ are defined
for all $t \geq 0$, and they are also uniformly bounded on
$\RR_+$, since $\widetilde{X}(t)$ evolves in the compact set
$K \subset \mathcal{W}$, the map $\mathcal{W} \ni \widetilde{X}
\mapsto \widetilde{X}(\widetilde{X}-I)(\widetilde{X}+I)^{-3} \in
\CC^{n \times n}$ is smooth, and $\widetilde{S}_k(t)$ is
$T$-periodic (and thus is bounded). Note that $f_k > 0, k=1,
\ldots, m$, can be seen as  ``feedback gains'' in the expression of
$\widetilde{u}_k$ above.

The corresponding deterministic control laws are given as
\begin{equation}\label{deterministiclaws}
    u_k(X,t) = \overline{u}_k(t) + \widetilde{u}_k(X,t) = \sum_{\ell=1}^{M} a_{k,\ell} \sin(\ell
\omega t) + f_k \mbox{Tr}
\left(X (\overline{X}^{\dag}(t) X - I)(\overline{X}^{\dag}(t) X + I)^{-3}\overline{X}^{\dag}(t) S_k \right)
\end{equation}

\subsection{Stochastic Control Laws}\label{stochaticlawsalgorithm}

At the time instants $t=j T, j \in \NN$, one chooses the amplitudes $a_{k,\ell}^j$ of
the reference controls $\overline{u}_k(t)$ in \eqref{refcon} for $t
\in [jT,(j+1)T)$ in an independent stochastic manner following a
uniform distribution on the interval
$[-\frac{a_{max}}{2},\frac{a_{max}}{2}]$, where $a_{max} > 0$
is arbitrarily fixed. More precisely, assume the Lebesgue probability
measure on $\mathbf{A}=[-\frac{a_{max}}{2},\frac{a_{max}}{2}]^{mM}$
(with the Borel algebra). An element $\mathbf{a}_j=(a_{k,\ell}^j)=
(a_{1,1}^j,\dots,a_{1,M}^j, \dots, a_{m1}^j,\dots,a_{mM}^j) \in
\mathbf{A}$ gathers the random amplitudes $a_{k,\ell}^j \in
[-\frac{a_{max}}{2},\frac{a_{max}}{2}]$ of the reference controls
$\overline{u}_k(t)$ in \eqref{refcon} for $t \in [jT,(j+1)T)$. Fix
any i.i.d.\ (independent and identically distributed) random vectors
$\mathbf{a}_j=(a_{k,\ell}^j)$:~$ \NN \rightarrow \mathbf{A}
\subset \RR^{mM}$ having a uniform distribution on
$\mathbf{A}$. The resulting bounded continuous reference controls
for $k=1, \ldots , m$ are
\begin{equation}\label{refconrandom}
  \overline{u}_k(t) = \sum_{\ell=1}^{M} a_{k,l }^j \sin(\ell
\omega t), \quad t \in [jT,(j+1)T), \; j \in \NN.
\end{equation}

The corresponding stochastic control laws are given as
\begin{equation}\label{stochasticlaws}
    u_k(X,t) = \overline{u}_k(t) + \widetilde{u}_k(X,t) = \sum_{\ell=1}^{M} a_{k,l }^j \sin(\ell
\omega t) + f_k \mbox{Tr}
\left(X (\overline{X}^{\dag}(t) X - I)(\overline{X}^{\dag}(t) X + I)^{-3}\overline{X}^{\dag}(t) S_k \right).
\end{equation}

\subsection{Application to Quantum Control}\label{qma}

This section exhibits the simulation results obtained when the deterministic and stochastic control laws were applied in the generation
of the C--NOT (Controlled-NOT) quantum logic gate on $\mbox{U}(4)$   and involving two qubits. Thus the underlying Hilbert space corresponds to the tensor product $\CC^2\otimes\CC^2 \equiv \CC^4$, each $\CC^2$ being the Hilbert space of one qubit. We assume here that the dynamics is governed by the

 following driftless  system of the form \eqref{cqs} with $n=4$ and $m=6$,
\[
     \dot{X}(t) =  - \iota \sum_{k=1}^6 u_k(t) H_k X(t),
\]
where
\[
    H_1 = \sigma_x \otimes I_2, H_2 = I_2 \otimes \sigma_x, H_3 = \sigma_y
\otimes I_2, H_4 = I_2 \otimes \sigma_y, H_5 =  \sigma_x \otimes \sigma_x
+ \sigma_y \otimes \sigma_y + \sigma_z \otimes \sigma_z,
H_6 = I_2 \otimes I_2,
\]
$\sigma_x, \sigma_y, \sigma_z \in \CC^{2 \times 2}$ are the usual Pauli matrices, $\otimes$ is the tensor product, and $I_2 \in \CC^{2 \times 2}$ is the 2-identity matrix. It can be shown that this system is controllable (and only Lie brackets of up to length 3 are required).
The aim is to generate two distinct goal matrices:
\[
X_{1_{goal}} = \exp(\iota \pi / 2) \left[
\begin{array}{cccr}
  1 & 0 & 0 & 0 \\
  0 & 1 & 0 & 0 \\
  0 & 0 & 0 & 1\\
  0 & 0 & 1 & 0
\end{array}
 \right], \qquad
X_{2_{goal}} = \left[
\begin{array}{ccrc}
  1 & 0 & 0 & 0 \\
  0 & 0 & -1 & 0 \\
  0 & 1 & 0 & 0\\
  0 & 0 & 0 & 1
\end{array}
\right].
\]
The first goal propagator, up to the irrelevant global phase $e^{i\pi/2}$,   corresponds to a
C--NOT  gate which is fundamental in quantum computation \cite{NieChu00}. The second one was
chosen for academic purposes. Note that $X_{1_{goal}}, X_{2_{goal}} \in \mathcal{W}$. Some $\mbox{MATLAB}^\circledR$
simulations have been done with this system, implementing both
deterministic \eqref{deterministiclaws} and stochastic \eqref{stochasticlaws} control laws. One has chosen $M=4$, $a_{max} =
0.25$, $f_k= M  a_{max} = 0.75$ and $T=25 $ for all simulations.
For the simulation of Figures \ref{error} and \ref{loglyap}, the final time is
$t_f= 10 T$ and the goal matrix is $X_{1_{goal}}$.
Figure~\ref{error} shows: (above) the convergence of the Frobenius norm $\|
X(jT) - X_{goal}\|$ to zero as $j \to \infty$ for both methods; and (below) the
input norm $\| u(t) \| = \|(u_1(t), \dots, u_6(t)) \|$.
Figure~\ref{loglyap} presents the natural logarithm of the Lyapunov function for
$t=j T, j \in \NN$, for the deterministic and stochastic cases.  One sees that
the stochastic laws provide a remarkably better overall performance than the
deterministic one. Note that
the fact that both curves tends to straight lines indicates that the convergence is  exponential for both methods.

\begin{figure}[!htbp]
    \centerline{\includegraphics[width=\textwidth]{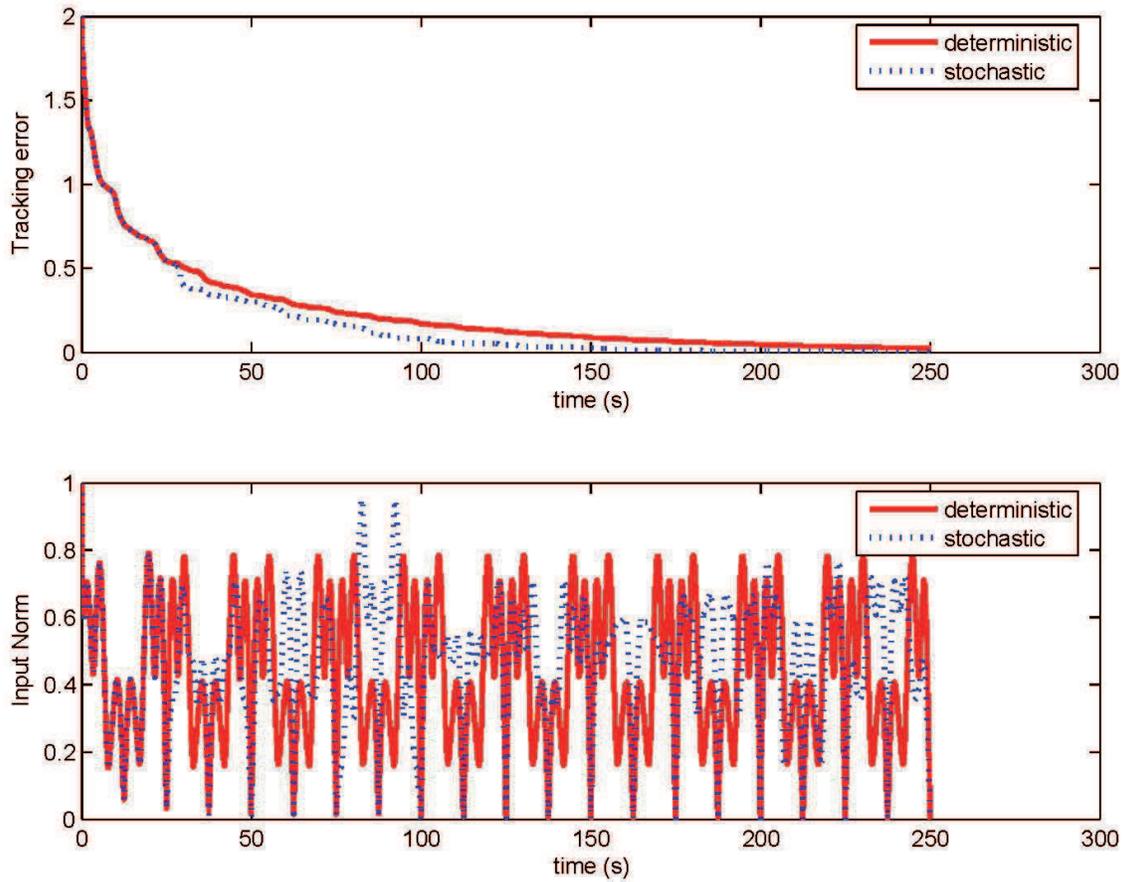}}
    \caption{Convergence of the norm of the tracking error to zero and input norm. For the stochastic control law,  the first choice of amplitudes $a^j_{k,l}$ for $j=0$ coincides with the amplitudes $a_{k,l}$  taken in deterministic control law; this explains the perfect overlaps between the solid and dotted curves  for $t$ between $0$ and $T=25$.}
     \label{error}
\end{figure}

\pagebreak

\begin{figure}[!htbp]
   \centerline{\includegraphics[width=\textwidth]{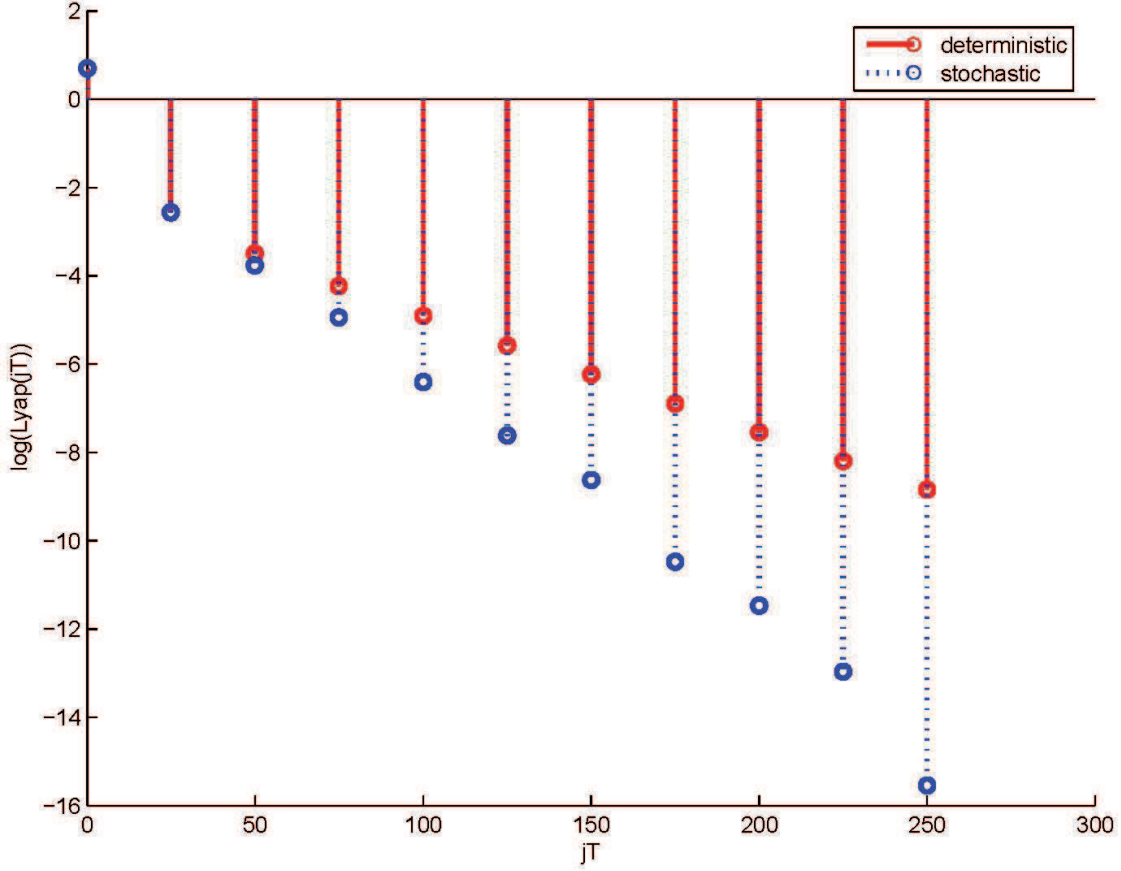}}
    \caption{Logarithm of the Lyapunov function at each sampling time  $t=jT$ with $T=25$.}\label{loglyap}
\end{figure}

Recall the notation
$\mathbf{a} = \{ a_{k,\ell} \; \in \RR : \; k=1, \ldots, m, \; \ell=1, \ldots ,
M\} \in \RR^{mM}$.
To make a fair comparison, one has made a total of 50 simulations (indexed by $p$)
for $t_f = 25 T= 625 $,
for both goal matrices $X_{1_{goal}}$ and $X_{2_{goal}}$, but now
with (fixed) different choices $\mathbf{a}^p$ of the set of amplitudes $\mathbf{a}$
for the deterministic strategy. In order to specify the
$\mathbf{a}^p$ in the $p$-th deterministic simulation, one has
chosen a random (uniformly distributed) $\mathbf{a}^p \in
[-a_{max},a_{max}]^{m M}$. It is important to point out that the $p$-th
deterministic simulation for both goal propagators uses the same
fixed choice $\mathbf{a}^p$. The control and simulation parameters are the same as before.
Figure~\ref{montecarlo} compares the
obtained results. At the top, one sees the
results for $X_{goal}=X_{1_{goal}}$ and, at the middle, the results for $X_{goal}=X_{2_{goal}}$. Notice
that the results of the deterministic strategy
depend strongly on the chosen goal matrix for a fixed $\mathbf a$. At the
bottom of Figure~\ref{montecarlo}, one sees a zoom view for
$X_{goal}=X_{1_{goal}}$. Observe that the worst result of
the stochastic strategy is of the order of the best result of the
deterministic strategy. Simulations not shown here have indicated
that, as $t_f/T$ gets greater, better is the performance of the
stochastic strategy when compared to the deterministic one.

\begin{figure}[!htbp]
    \centerline{\includegraphics[width=\textwidth]{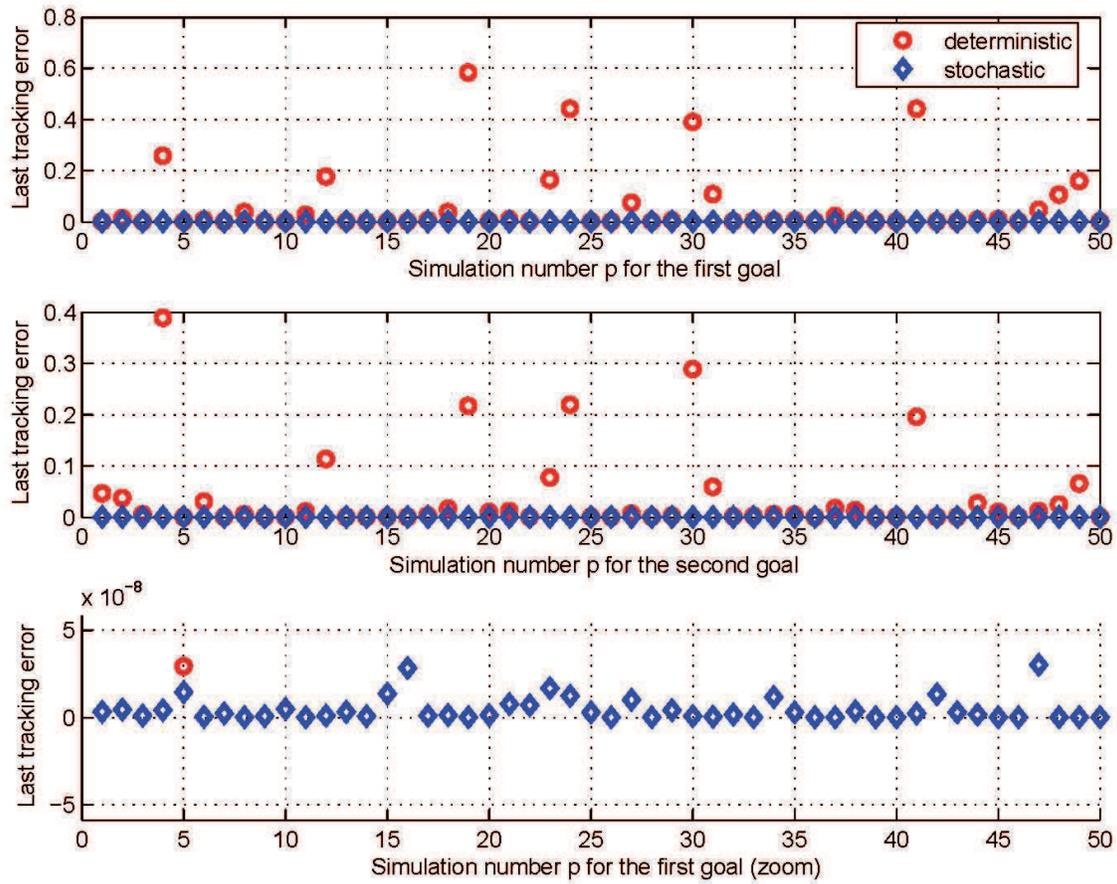}}
    \caption{Error between $X$ and $X_{\mu,goal}$  at $t_f=25 T$ for 50 Monte Carlo simulations:  top  corresponds to  $\mu=1$, middle to $\mu=2$; bottom  is a zoom for $\mu=1$.  }\label{montecarlo}
\end{figure}

\clearpage

\subsection{Main Results}\label{deterministicstochasticresults}

The next theorem shows that, for $M > 0$ big enough, a random (but
fixed) choice of the coefficients $a_{k,\ell}, k=1, \ldots, m, \ell=1,
\ldots , M$ in \refeq{refcon} gives a local solution to the deterministic $T$-sampling stabilization problem with probability one.
\begin{theorem}\label{t1}(Local Deterministic Method)
Assume that \eqref{cqs} is controllable and that its initial
condition is $X(0)=X_0=I$. Let $T > 0$, $X_{goal} \in \mathcal W \subset
\mbox{U}(n)$, where $\mathcal W$ is as in \refeq{eW}. Choose
any nonzero $f_1, \dots, f_m \in \RR$. Fix a integer $M >0$ and a choice of
$\bar {\mathbf a}= (\bar{a}_{k,\ell} \in \RR \; : \; k=1, \ldots,
m, \ell=1, \ldots, M) \in \RR^{m M}$ in \refeq{refcon}. One says that
this choice of $\bar{\mathbf a}$ is \emph{admissible} if it solves the
$T$-sampling stabilization problem, that is \refeq{ctez-i} holds for the
closed-loop system
(\ref{tesys},\ref{feedback}). Let ${\mathcal A}^c \subset
\RR^{m  M}$ be the set of all non-admissible $\bar {\mathbf a}$.
Then, there exists $M > 0$ big enough such that ${\mathcal A}^c$ is
closed in $\RR^{m M}$ with zero Lebesgue measure.
\end{theorem}

The stochastic version of the result above is:

\begin{theorem}\label{mainstochasticresult}(Local Stochastic Method)
Assume that \eqref{cqs} is controllable and that its initial
condition is $X(0)=X_0=I$. Let $T > 0$, $X_{goal} \in \mathcal W \subset
\mbox{U}(n)$, where $\mathcal W$ is as in \refeq{eW}. Choose
any nonzero $f_1, \dots, f_m \in \RR$. Fix
any $a_{max} > 0$. Then, there exist $\overline{M}
> 0$ big enough such that, for every $M \geq
\overline{M}$, the stochastic control laws \eqref{stochasticlaws} assure that $
  \lim_{j \to \infty} X(jT) = X_{goal}$
almost surely.
\end{theorem}

The exponential convergence for both methods is established by:

\begin{theorem}\label{expconv}
The convergence of $X(jT)$ towards $X_{goal}$ in Theorem~\ref{t1}
 is exponential. The almost sure convergence of $X(jT)$ towards $X_{goal}$ in Theorem~\ref{mainstochasticresult}
 is also exponential.
\end{theorem}

It is important to point out that Remark~\ref{worstdirec} in Section~\ref{proofexpconv} gives a clue of why the
convergence of the stochastic method allows a greater Lyapunov exponent.

The results below solve the $T$-sampling stabilization problem for
an arbitrary initial condition:
\begin{corollary}\label{c1} If system \refeq{cqs} is controllable,
then the $T$-sampling stabilization problem is
solvable for any initial condition $X(0)=Y_0 \in \mbox{U}(n)$ and any
goal state $Y_{goal} \in \mbox{U}(n)$ such that $Y_{goal} Y_0^\dag
\in \mathcal W$.
\end{corollary}
\begin{proof}
Since \eqref{cqs} is a right-invariant system, one has the following well-know property. For any fixed $Z \in \mbox{U}(n)$
and any fixed set of piecewise-continuous inputs $u_k(t)$, $k =1, \dots, m$, then $X(t)$ is a
solution of \refeq{cqs} if and only if $Y(t) = X(t) Z$ is a solution
of \refeq{cqs} with the same applied inputs. Now, solve the problem
 for $X(0) = X_0 = I$ and $X_{goal} = Y_{goal} Y_0^\dag \in \mathcal{W}$ using Theorem \ref{t1} (resp. Theorem~\ref{mainstochasticresult})
 and then apply the corresponding inputs \eqref{deterministiclaws} (resp.\ \eqref{stochasticlaws}) to system
\refeq{cqs}. Note that this corresponds to choose $Z = Y_0$.
\end{proof}

\begin{remark} If one accepts a global phase change on $X_{goal}$ of the form $X_{goal} \in \mbox{U}(n) \mapsto \exp(\iota \phi)X_{goal} \in \mathcal{W} \subset \mbox{U}(n)$, which is immaterial for quantum systems, then it is easy to show that, as the set of eigenvalues of $X_{goal}$ is a discrete subset of $\CC$, one may choose
a convenient $\phi$ such that Theorem~\ref{t1} (deterministic) and Theorem~\ref{mainstochasticresult} (stochastic)   globally solve the $T$-sampling stabilization problem on $\mbox{U}(n)$.
\end{remark}

It could be interesting in some situations to assure global convergence
without accepting a global phase transformation.
The previous result implies\footnote{Another advantage of
this strategy in two steps was verified
 in simulations. In some cases it may generate smaller inputs when
 compared with the one step procedure, even it is combined with
 global phase change.} that one may introduce the following global
strategy for the $T$-sampling stabilization of any initial state $X_0$ towards any goal state $X_{goal}$ that do not obey the assumption of
Corollary \ref{c1}.

\begin{theorem}
 \label{tSteer}(Global Deterministic or Stochastic Method)
Given arbitrary $X_0, X_{goal} \in \mbox{U}(n)$, then one may always
construct $X_1 \in \mbox{U}(n)$ in a way that both $X_1 X_0^\dag$ and $X_{goal}
X_1^\dag$ are in $\mathcal W$. Then one may implement the following piecewise-smooth
controls law in two steps:
\begin{itemize}
 \item \textbf{Step one.} Apply the control law of Corollary \ref{c1} to $Y_0=X_0$ and
  $Y_{goal} = X_1$. Then there exists $L \in \NN$ big enough such
  that $X_{goal} \left. X(t)^{\dag}\right|_{t=LT} =  X_{goal} X(L T)^{\dag} \in \mathcal
  W$.
  \item \textbf{Step two.} At the instant $\overline t =L T$, switch
  the control law to the one of
  Corollary \ref{c1} for $Y_0 = X(L T)$ and $Y_\infty =
  X_{goal}$.
\end{itemize}
This control policy then solves the $T$-sampling stabilization problem.
\end{theorem}

\begin{proof} If $W=X_{goal} X_0^\dag \not\in \mathcal W$, then one may
 write $W = U^\dag D U$, where $U$ is a unitary matrix and $D$ is a diagonal matrix
 whose diagonal entries $d_i= \exp(\iota \theta_i)$, $i = 1, \dots, n$, are the eigenvalues of $W$.
 One can always assume that $\theta_i \in (-\pi, \pi]$, $i=1, \dots, n$.
 Take $W_1 = U^\dag D_1 U$, where $D_1$ is a diagonal matrix with diagonal entries
 given by $d^1_i= \exp(\iota \theta_i/2), i = 1, \ldots, n$. It is clear
 that $W = W_1 W_1$ and $W_1 \in \mathcal W$. Let $X_1 = W_1 X_0$.
 By construction, one has $X_1 X_0^\dag=W_1 \in \mathcal W$ and
 $X_{goal} X_1^\dag = W W_1^\dag = W_1 \in \mathcal W$.
 In the first step of our control law, one has
 $\lim_{j \rightarrow \infty} X(j T) = X_1$.
 By the well-known property of the continuity of eigenvalues, for $ L \in \NN$
 big enough, one has $X_{goal} X(L T)^\dag \in \mathcal W$.
Then, at $\overline{t} = L T$, if one switches the control law to the one of
Corollary \ref{c1} with $Y_0 = X(L T)$ and $Y_{goal} = X_{goal}$,
the claimed properties hold.
\end{proof}

\subsection{Proof of the Main Deterministic Result}\label{mrtd}

The technical details involved in the proof of Theorem~\ref{t1} are given in the sequel.

Let $u_k$:~$\RR \rightarrow \RR$, $k = 1, \ldots, m$, be an arbitrary set of smooth controls and fix any initial condition $X(0) \in \mbox{U}(n)$ in \eqref{cqs}. Define
\begin{equation}\label{rss-m}
    \begin{array}{l}
        A(t) = \sum_{k=1}^{m} u_{k}(t) S_k \in \mathfrak{u}(n), \medskip \\
        B^0_k(t) = S_k X(t), \medskip \\
        B^{j+1}_k(t) = \dot{B}^{j}_k(t) - A(t)B^j_k(t),
    \end{array}
\end{equation}
for each $j \in \NN$, $k= 1,\dots, m$, $t \in \RR$, where $X(t)$ is the corresponding solution of \eqref{cqs}.
One remarks that the linearized system of (\ref{cqs}) along the
trajectory  $(X(t), u_k(t), k=1\dots, m , t \in
\RR)$ is given by the time-varying linear control system
$\dot{X}_\ell(t) = A(t) X_\ell(t) + \sum_{k=1}^m w_k(t) B^0_k(t)$, where
$X_\ell \in \CC^{n \times n}$ is the state and $w_k \in \RR$
are the controls \cite{Cor94, Cor07}. Let $[E, F] = E F - F E$ be the usual commutator of the matrices $E, F \in
\CC^{n\times n}$. Define
\begin{equation}
 \begin{array}{rcl}\label{defCkj}
  C_k^0(t) & = & S_k, \\
  C_k^{j+1}(t) & = & \dot{C}_k^{j}(t) + \left[ {C}_k^{j}(t),  A(t) \right],
 \end{array}
\end{equation}
for $k =1, \ldots m$, $j \in \NN$, $t \in \RR$. It is straightforward to conclude that $C^j_k(t) X(t)=B^j_k(t)$. In particular,
$C^j_k(0)=B^j_k(0)$ when $X(0) =I$.

If the smooth controls $u_k$ are $T$-periodic and
  \begin{equation}
  \label{sinus}
u_k(T-t) = -u_k(t), \quad \mbox{for } t \in \RR, \; k=1, \dots, m,
   \end{equation}
then the solution $X(t)$ of \refeq{cqs} is also $T$-periodic (see \cite{Cor07} for details). It is clear that $\overline{u}_k$ in \eqref{refcon} is $T$-periodic
and satisfies \eqref{sinus}.
\begin{definition}\label{rsd}
Let $T >0$. System (\ref{cqs}) is $T$-\emph{regular}  when there
exist smooth $T$-periodic inputs $\widehat{u}_k$:~$\RR
\rightarrow \RR$, $k=1,\dots, m$, called
here \emph{Coron reference controls}, such that \refeq{sinus} holds
and the corresponding $T$-periodic solution $\widehat X$:~$\RR \rightarrow
\mbox{U}(n)$ of (\ref{cqs}) with initial condition $\widehat
X(0)=I$ and $u_k = {\widehat{u}}_k$ satisfies
\begin{align}\label{rss}
    \mathfrak{u}(n) = \mbox{span} \{ B^j_k(0), \mbox{ for all } k=1,\dots, m, \; j \in \NN \},
\end{align}
where the $B_k^j$'s are as in \eqref{rss-m} for the inputs $\widehat{u}_k$. Note that $C_k^j(0)=B_k^j(0)$ for such $\widehat{u}_k$.
When (\ref{cqs}) is $T$-regular for every $T > 0$, one simply says
that system (\ref{cqs}) is \emph{regular}.
\end{definition}

\begin{remark}\label{atp}
If the control problem is solved for some $T > 0$, it will then be
solved for all $\overline{T} > 0$. This relies on standard
time-scaling arguments. In fact, if $X_a(t)$ is a solution of
\refeq{cqs} when the inputs are given by $u^a_k(t), k = 1, \ldots,
m$, then for every $\alpha > 0$, one has that $X_b(t) =
X_a(\alpha t)$ is a solution corresponding to the inputs $u_k^b(t) =
\alpha u_k^a(\alpha t), k = 1, \ldots, m$.
\end{remark}

By redefining
the Coron reference controls $\widehat{u}_k$ as described in Remark~\ref{atp}, it is clear
that system (\ref{cqs}) is regular whenever it is $T$-regular for
\emph{some} $T > 0$. The next theorem, which relies on the results of the Return
Method developed in \cite{Cor94}, establishes that the regularity
of (\ref{cqs}) is in fact equivalent to its controllability on
$\mbox{U}(n)$.

\begin{theorem}\label{crm}
The following assertions are equivalent for system \eqref{cqs}:
\begin{enumerate}
 \item System (\ref{cqs}) is regular;
 \item System (\ref{cqs}) is $T$-regular for some $T > 0$;
 \item $\mbox{Lie}\{S_1, \dots, S_m \}=\mathfrak{u}(n)$.
\end{enumerate}
\end{theorem}
\begin{proof}
It was seen above that $1$ and $2$ are equivalent. That $2
\Rightarrow 3$ follows from the presentation in \cite[pp.\ 360,
361, 362]{Cor94}. Finally, $3 \Rightarrow 1$ is established in
\cite[Remark 3.1 on p.\ 377]{Cor94} (see also \cite[pp.\
187--192, 296--298]{Cor07}).
\end{proof}

When a given set of inputs $u_k(t)$, $k=1, \dots, m$, are Coron reference controls, that is
\refeq{rss} holds and $C_k^j(0)=B_k^j(0)$, there must exist integers $k_1, k_2, \ldots k_d
 \in \{1, \ldots, m\}$ and $j_1, j_2, \ldots, j_d \in \NN$ such that
\begin{equation}
 \label{eGenerates}
 \mathfrak{u}(n) = \mbox{span} \left\{ C_{k_h}^{j_h}(0) : h=1, 2, \ldots, d \right\},
\end{equation}
where $d \triangleq \dim \mbox{U}(n)$.
However, note that the
 $C_{k}^j(0)$'s may be computed for any chosen smooth $T$-periodic controls
in \eqref{cqs} (with initial condition $I$ at $t=0$) obeying \refeq{sinus}, but
\emph{a priori}
it is not assured that \refeq{eGenerates} will hold. When \refeq{eGenerates}
is met, then the chosen inputs are indeed Coron reference controls.

The following theorem shows that one may generate  Coron
reference controls with probability one by randomly choosing
the coefficients $a_{k,\ell} \in \RR$ in \eqref{refcon}. Note that its hypothesis is always met whenever system \eqref{cqs}
is controllable, cf. Theorem~\ref{crm}.

\begin{theorem}\label{tProb1} Let $T > 0$. Assume that there exists one set of
Coron reference controls $\overline{u}_k$, $k=1, \dots, m$, for which \refeq{eGenerates} holds.
Let $J = \max \{ j_1, \dots, j_d \}$, where $j_h$, $h=1,\dots,d$, are as in \refeq{eGenerates}. Take $M =  J/2$ (integer
division).
Let $\mathcal O \subset \RR^{m M}$ be the set of all vectors $\bar{\mathbf a} = (\overline{a}_{k,\ell} \in \RR \; : \; k=1, \ldots,
m, \ell=1, \ldots, M) \in \RR^{m M}$ for which \refeq{eGenerates} is met when $\overline{u}_k$, $k=1,\dots,m$, are given by
\refeq{refcon} and $\overline{X}(t)$ is the corresponding $T$-periodic trajectory of \eqref{cqs} with $\overline{X}(0)=I$ and $u_k = \overline{u}_k$.
Then, $\mathcal O$  is a dense open set in $\RR^{m M}$ and its complement $\mathcal{O}^c$ has zero Lebesgue measure.
\end{theorem}

\begin{proof}
 See Appendix \ref{aProb1}.
\end{proof}

From now on, $\overline{u}_k$, $k=1, \dots, m$, will denote a choice
of Coron reference controls and $\overline{X}_I(t)$ will be the
corresponding $T$-periodic trajectory of \eqref{cqs} with
$\overline{X}_I(0)=I$ and $u_k=\overline{u}_k$. Given any desired
goal state $X_{goal} \in \mbox{U}(n)$, define
\begin{align}\label{rtd-r}
    \overline{X}(t) = \overline{X}_I(t) X_{goal}, \quad t \in \RR.
\end{align}
Note that $\overline X(t)$ is the resulting $T$-periodic reference trajectory of \eqref{rsys} with $\overline X(0) = X_{goal}$ and the same controls. Our main stability
result is now presented.

\begin{theorem} \label{tMain} Choose $T > 0$. Let $X_0=X(0)=I$ and $X_{goal} \in \mathcal W$.
Let $\overline{u}_k, k =1, \ldots, m$, be Coron reference controls and
let $\overline X(t)$ be the $T$-periodic reference trajectory of \refeq{rsys} given in \eqref{rtd-r}.
Then the control laws $\widetilde{u}_k$ in \refeq{feedback} solve the $T$-sampling stabilization problem,
that is \refeq{ctez-i} holds for the closed-loop system (\ref{tesys},\ref{feedback}).
\end{theorem}

\begin{proof} The idea of the proof is the application of Lasalle's invariance principle,  Theorem \ref{tLaSalle} in Appendix
\ref{aLasalle}. For this, consider the closed-loop system (\ref{tesys},\ref{feedback}) with
state evolving on $S = \{ \tilde X \in \CC^{n\times n} \; | \; \det
(I+X) \neq 0\}$. In this case, one is regarding a system evolving on
an open set of a (complex) Euclidean space\footnote{One could
restate Theorem \ref{tLaSalle} for smooth manifolds without any
problem. However the notation becomes awful, and a notion of
distance must be chosen,  for instance the one that $\mbox{U}(n)$
inherits from $\CC^{n \times n}$.}. Recall that the set $K \subset
\mathcal W \subset S$ defined by \refeq{eK} is a compact positively invariant
set\footnote{At this moment, one is considering that $K$ is compact set with respect to the
topology of the Euclidean space $\CC^{n \times n}$. As $\mbox{U}(n)
\subset \CC^{n \times n}$ is an embedded manifold, the topology of
$\mbox{U}(n)$ is equivalent to the subspace topology induced by
$\CC^{n \times n}$, and so this distinction is of minor
importance.}.

Then, for any $\widetilde X(0) = \widehat{X} \in K$,
 the solution $\widetilde X(t)$ of
(\ref{tesys},\ref{feedback}) is defined for all $t \geq 0$
and remains inside  $K$. Since (\ref{tesys},\ref{feedback}) is a
$T$-periodic system, then by Theorem~\ref{tLaSalle} it suffices to show
that the set
\[
  E = \{ (\widehat X,t) \in   K \times \RR^+ \; | \; \dot {\mathcal V}(\widehat{X}  ,t) = 0  \}
\]
does not contain any nontrivial solution $(\widetilde X(t), t)$,  $t\geq
0$, of (\ref{tesys},\ref{feedback}), that is only the trivial
solution $(I, t)$, $t \geq 0$ is contained in $E$.  For that, according to \eqref{negdef},  $\dot {\mathcal V}(\widetilde X(t),t) =0$ for all $t \geq 0$ along a solution implies that the control law
\refeq{feedback} is identically zero. Hence, such a solution must be
an equilibrium point of \refeq{tesys}, that is $\widetilde X(t)$
must identically equal to its initial condition $\widetilde X(0) = \widehat X$. Now, in
\refeq{negdef}, let
 \begin{equation}
  \label{Zhat}
\widehat{Z}=\widehat{X}(\widehat{X}-I)(\widehat{X}+I)^{-3}.
 \end{equation}
Suppose that $\mbox{Tr}\left(\widehat{Z}\overline{X}^{\dag}(t) S_k
\overline{X}(t)
\right)=\mbox{Tr}\left(\widehat{Z}\overline{X}^{\dag}(t) C_k^0(t)
\overline{X}(t) \right)=0$ for $t \in [0,T]$, $k=1, \ldots , m$.
It will be shown by induction that, for $k=1, \ldots , m$, one has
\begin{equation}\label{induction}
  \mbox{Tr}\left(\widehat{Z}\overline{X}^{\dag}(t) C_k^j(t) \overline{X}(t)
\right) = 0, \quad \mbox{for } t \in [0,T], \; j \in \NN.
\end{equation}
This is true for $j=0$ and assume it holds for a fixed $j \in \NN$.
Taking the derivative at both sides of \refeq{induction} and using \eqref{defCkj}, one
obtains
\[
\begin{array}{l}
\mbox{Tr}\left(\widehat{Z} \left[\dot{\overline{X}}^{\dag}(t)
C_k^j(t) \overline{X}(t) \right] \right) +
\mbox{Tr}\left(\widehat{Z} \left[ \overline{X}^{\dag}(t)
\dot{C}_k^j(t) \overline{X}(t) \right] \right)    +
\mbox{Tr}\left(\widehat{Z} \left[
\overline{X}^{\dag}(t) C_k^j(t) \dot{\overline{X}}(t) \right] \right)
 \\
 =  \mbox{Tr}\left(\widehat{Z}\left( \overline{X}^{\dag}(t)
\dot{C}_k^j(t) \overline{X}(t) \ \right) \right)   +
\mbox{Tr}\left(\widehat{Z}\left( \overline{X}^{\dag}(t) \left[
C_k^j(t), \sum_{\ell=1}^m u_\ell^T(t) S_\ell \right] \overline{X}(t)
\right) \right)
 \\
  = \mbox{Tr}\left(\widehat{Z} \overline{X}^{\dag}(t) C_k^{j+1}(t)
\overline{X}(t) \right) = 0, \quad \mbox{for } t \in [0,T].
\end{array}
\]
Hence \refeq{induction} holds. Taking $t=0$, one gets
\[
\mbox{Tr}\left(\widehat{Z} \overline{X}(0)^{\dag} C_k^j(0)
\overline{X}(0) \right) = \mbox{Tr}\left(\widehat{Z}
X_{goal}^{\dag} C_k^j(0) X_{goal} \right) = 0.
\]
By \refeq{rss} and from the fact that $C_k^j(0)$=$B_k^j(0)$, to
conclude the proof it suffices to show that $\mbox{Tr} \left(
\widehat Z \Sigma \right)=0$, for $\widehat Z$ given by \refeq{Zhat}
and all $\Sigma \in \mathfrak{u}(n)$, implies that $\widehat X = I$.
Recall that one may write $\widehat X = U^\dag D U$, where $U$
is unitary, $D =\diag( \exp{(\iota a_1)}, \ldots,
\exp{(\iota a_n)})$ is a diagonal matrix, and $d_i=\exp{(\iota a_i)}, i=1,
\ldots, n$, are the eigenvalues of $\widehat X$ with $a_i \in (-\pi,
\pi)$. Note that $\det \widehat X =1$ implies that $\sum_{i=1}^{n}
a_i=0$ (mod $2\pi$). Simple computations using the identities $\exp{(\iota a)} - 1 =
2\iota  \sin(a/2) \exp{(\iota a/2)}$ and $\exp{(\iota a)} + 1 = 2 \cos(a/2)
\exp{(\iota a/2)}$ results in $\widehat Z = U^{\dag} D_1 U$, where $D_1 =
\diag(\iota  \alpha(a_1), \ldots ,  \iota \alpha(a_n))$ with
$\alpha(a_i) = \tan(a_i/2)/ [2\cos(a_i/2)]^2$. It is easy to
show\footnote{Simple computations show that $\lim_{a \to \pm \pi}
\alpha(a) = \pm \infty$ and $\frac{d\alpha}{da}(a) =
\frac{3- 2\cos(a/2)}{8 [\cos(a/2)]^4}
>0$, for $a \in -(\pi, \pi)$.} that the function $(-\pi, \pi) \ni a \mapsto
\alpha(a) = \tan(a/2)/ [2\cos(a/2)]^2 \in \RR$ is injective, it is surjective (onto $\RR$), and $\alpha(0) =0$. Now, taking the
matrices in $\mathfrak{u}(n)$ of the form $U^{\dag} \Xi U$ with
 \begin{equation}
  \label{sunchoice}
\Xi = \diag( 0, \ldots, 0,  0, \ldots, 0, -\iota, 0, \ldots, 0) \in \mathfrak{u}(n)
 \end{equation}
and using the invariance of the trace, one obtains that all the
diagonal entries of $D_1$ are zero. As the map $\alpha$ is
injective, it follows that $a_i= 0, i=1, \ldots, n$, and this concludes
the proof.
\end{proof}

\begin{proof} (\emph{of Theorem~\ref{t1}})
A straightforward consequence of Theorem
\ref{crm}, Theorem \ref{tProb1} and Theorem \ref{tMain}.
\end{proof}

\subsection{Proof of the Main Stochastic Result}\label{proofstochastic}

This subsection presents the proof of Theorem~\ref{mainstochasticresult}. First of all, the tracking error $\widetilde{X}(t)$ is sampled with
sampling period $T$ in order to apply stochastic Lyapunov stability
results that will assure that $\lim_{j \to \infty}
\widetilde{X}(jT)=I$. Now, for each sampling interval $[jT,(j+1)T)$, $j
\in \NN$:
\begin{itemize}
    \item One considers the reference controls as in \eqref{refconrandom};
    \item One defines similarly the reference trajectory $\overline{X}(t)$,
the tracking error $\widetilde{X}(t)$ and the feedbacks
$\widetilde{u}_k(t)$ by taking $\overline{u}_k(t)$ as in
\eqref{refconrandom};
    \item One lets $\widetilde{X}_j=\widetilde{X}(jT) \in \mathcal{W}$ for $j
\in \NN$.
\end{itemize}

The vector field of the closed-loop system
(\ref{tesys},\ref{feedback}) with the choice \refeq{refcon}
depends smoothly on the reference controls parameters $a_{k,\ell}$.
Recall that
$\widetilde{S}_k(t)=\overline{X}^{\dag}(t)S_k\overline{X}(t)$, where
the reference trajectory $\overline{X}(t)$ is the solution of
\eqref{rsys} with the reference controls $\overline{u}_k(t)$ in
\eqref{refcon}. Let $\Lambda$:~$\RR \times \RR \times
\mathcal{W} \times \RR^{mM} \rightarrow \mathcal{W}$ be the
($\mathbf{a}$-parameter dependent) smooth global flow of the closed-loop system
(\ref{refcon},\ref{tesys},\ref{feedback}). This means that
$\widetilde{X}(t)=\Lambda(t,t_0,\widehat{X},\mathbf{a}) \in
\mathcal{W}$,
 is the  solution of
system (\ref{tesys},\ref{feedback}) with the choice
\eqref{refcon} and with initial condition
$\widetilde{X}(t_0)=\widehat{X} \in \mathcal{W}$ at $t=t_0$ and
reference controls parameters $\mathbf{a}=(a_{k,\ell}) \in
\RR^{mM}$. In particular, the map $(\widetilde{X},\mathbf{a})
\ni \mathcal{W} \times \mathbf{A} \mapsto
\Lambda(T,0,\widetilde{X},\mathbf{a}) \in \mathcal{W}$ is
continuous. Since system
(\ref{refcon},\ref{tesys},\ref{feedback}) is $T$-periodic in
$t$, one has that $\Lambda(t+jT,jT,\widetilde{X},\mathbf{a}) =
\Lambda(t,0,\widetilde{X},\mathbf{a})$ for every $t \geq 0$, $j \in \NN$, $(\widetilde{X},\mathbf{a}) \in \mathcal{W} \times \mathbf{A}$
\cite[p.\ 143]{Vidyasagar93}.

The reasoning above implies that $\widetilde{X}_{j+1} =
\Lambda(T,0,\widetilde{X}_{j},\mathbf{a}_j)$, for $j \in \NN$,
where $\widetilde{X}_{0}=X_{goal} \in \mathcal{W}$. Consequently,
$\widetilde{X}_{j}$:~$\NN \rightarrow \mathcal{W}$
is a Markov chain (with respect to the
natural filtration and the Borel algebra on $\mathcal{W}$) because
$\mathbf{a}_j$, $j \in \NN$, are independent random vectors. Note
that \eqref{negdef} assures that $\mathcal{V}(\widetilde{X}_j)$, $j
\in \NN$, is a supermartingale. Define the continuous function
$Q$:~$\mathcal{W} \rightarrow \RR$ as
\begin{align}
\label{eDefQ} Q(\widehat{X}) & \triangleq \frac{-1}{a_{max}^{mM}}
\int_{\mathbf{A}} \left( \int_{0}^{T}
\dot{\mathcal{V}}\left(\Lambda(t,0,\widehat{X},\mathbf{a}),t \right)
\, dt \right) \, d\mathbf{a}
= \frac{-1}{a_{max}^{mM}} \int_{\mathbf{A}} \left( \int_{0}^{T}
\dot{\mathcal{V}}\left(\widetilde{X}(t),t \right) \, dt \right) \,
d\mathbf{a}.
\end{align}
By \eqref{negdef}, $Q$ is non-negative, and \eqref{feedback} gives
that $Q(I)=0$. For each $j \in \NN$, the conditional expectation of
$\mathcal{V}(\widetilde{X}_{j+1})$ knowing $\widetilde{X}_j$ is
denoted by $\mathbb{E}\left(\mathcal{V}(\widetilde{X}_{j+1}) /
\widetilde{X}_j\right)$. Since $\widetilde{X}_{j}$ is independent of
$\mathbf{a}_j$ and $\dot{\mathcal{V}}(\widetilde{X},t)$ is
$T$-periodic in $t$, one gets
\begin{align}
\mathbb{E}\left(\mathcal{V}(\widetilde{X}_{j+1}) /
\widetilde{X}_j\right) -
\mathcal{V}(\widetilde{X}_j)
= \frac{1}{a_{max}^{mM}} \left. \hspace{-2pt} \left[
\int_{\mathbf{A}} \hspace{-2pt} \left( \int_{jT}^{(j+1)T}
\hspace{-6pt}
\dot{\mathcal{V}}\left(\Lambda(t,jT,\widehat{X},\mathbf{a}),t
\right) \, \hspace{-2pt} dt \right) \hspace{-2pt} d\mathbf{a}
\right] \right|_{\widehat{X}=\widetilde{X}_j} =
-Q(\widetilde{X}_j). \label{eQQ}
\end{align}
Standard results on stochastic Lyapunov stability imply that
$\lim_{j \to \infty} Q(\widetilde{X}_j) = 0$ almost surely (see
Theorem~\ref{stochastic} in Appendix). One will show that the only
solution to $Q(\widehat{X}) = 0$ is $\widehat{X}=I$. This will prove
almost sure convergence of $\widetilde{X}_j=\widetilde{X}(jT)$
towards $I$ because $Q$ is continuous and $\widetilde{X}_j$ evolves
on the compact set $K=\{\widetilde{X} \in \mathcal{W} \; | \; 0 \leq
\mathcal{V}(\widetilde{X}) \leq \mathcal{V}(X_{goal}) \}$ in
$\mathcal{W}$.

Assume that $\widehat{X} \in \mathcal{W}$ is such that
$Q(\widehat{X}) = 0$. Then, $\dot{\mathcal{V}}(\Lambda(t,0,\widehat{X},\mathbf{a}),t)=0$
for $t \in [0,T]$ and $\mathbf{a}=(a_{k,\ell}) \in
\mathbf{A}$. In particular, \refeq{negdef} implies that
$\widetilde{u}_k(t)=0$ for $t \in [0,T)$. By relying on the ideas presented in the end of the proof Theorem~\ref{tMain}, one
concludes that $\widetilde X = I$. The preceding argument then proves
that $\lim_{j \to \infty} \widetilde X ( j T) = I$ almost surely, which completes the proof of Theorem~\ref{mainstochasticresult}.

\subsection{Proof of the Exponential Convergence Result}\label{proofexpconv}

Theorem~\ref{expconv} for  the deterministic strategy is immediate from the result given below
and standard Lyapunov stability results for discrete-time nonlinear systems.

\begin{lemma}
 \label{lExp}
Let $\| \widetilde W \| = \mbox{Tr}(\widetilde{W}^\dag \widetilde{W})$ stand for
the Frobenius norm of $\widetilde{W} \in \CC^{n \times n}$. Then:
\begin{itemize}
\item There exist $\epsilon_1, c_1, c_2 >0$ such that, for every $\widetilde X \in \mbox{U}(n)$
with $\| \widetilde X - I \| < \epsilon_1 $, then
\begin{equation}\label{ec1c2}
c_1 \| \widetilde X - I\|^2 \leq \mathcal V(\widetilde X) \leq c_2 \| \widetilde X - I \|^2.
\end{equation}

\item There exist $\epsilon_2, c_3, c_4 >0$ such that, for every $\widetilde X \in \mbox{U}(n)$
with $\| \widetilde X - I \| < \epsilon_2$, then
\begin{equation}\label{z2ineq}
 c_3 \| \widetilde X - I\|^2 \leq \|Z\|^2 \leq c_4 \| \widetilde X - I\|^2, \quad \mbox{where } Z=\widetilde{X}(\widetilde{X}-1)(\widetilde{X}+I)^{-3}.
\end{equation}

\item If $\mathbf a \in \mathbf A$ is admissible in the sense of Theorem~\ref{t1}, then there exists $M_{\mathbf a} > 0$ such that, for the trajectory $\widetilde X(t)$ of the closed-loop system (\ref{tesys},\ref{feedback}) with initial condition $\widetilde{X}(0)=X_{goal} \in \mathcal{W}$, there exists $L \in \NN$ big enough
such that, for all $j \geq L$, one has
    \begin{equation}
    \label {Ma}
    {\mathcal V}(\widetilde X((j+1)T)) -  {\mathcal V}(\widetilde X(jT)) \leq - M_{\mathbf a}\|\widetilde{X}(jT) - I \|^2.
    \end{equation}
\end{itemize}

\end{lemma}

\begin{proof} The first claim is an immediate consequence of \refeq{lyaptan} (see Appendix \ref{aCompact}), and the
second one is straightforward from the properties of the function $\alpha$ used in the proof of Theorem~\ref{tMain}.
Consider now that $X_{goal}$ is such that $\| X_{goal} - I \| \leq \min \{ \epsilon_1, \epsilon_2\}$. Then the proof of the third claim follows from the arguments below:
\begin{itemize}

\item[(a)] Let $\mathbf{G}$ be the set of skew-Hermitian matrices of unitary
(Frobenius) norm. It is clear that $\mathbf{G}$ is compact.
By using similar arguments as in the end of
the proof of Theorem~\ref{tMain}, it is easy to show that, for any fixed $Z^* \in \mathbf{G}$, one has
\[
\int_{0}^{T}
\sum_{k=1}^{m} f_k \mbox{Tr}^2(Z^* \widetilde{S}_k(t)) \, dt \triangleq P_{\mathbf a} (Z^*) > 0.
\]

\item[(b)]
As $P_{\mathbf a}(\cdot)$ is continuous in $Z^* \in \mathbf{G}$ and $\mathbf{G}$ is compact, this function admits a minimum $P_{\mathbf a}^* > 0$.
Thus, for any fixed $Z^* \in \mathbf{G}$,
\begin{equation}\label{constantineq}
\int_{0}^{T}
\sum_{k=1}^{m} f_k \mbox{Tr}^2(Z^* \widetilde{S}_k(t)) \, dt  \geq P_{\mathbf a}^* > 0.
\end{equation}

\item[(c)] Fix $X_{goal} \in \mathcal{W}$ and let $K$ be the set in \eqref{eK}. Recall that $K$ is positively invariant and compact, cf. Appendix~\ref{aCompact}. Consider the continuous map $K \ni\widetilde{X} \mapsto Z(\widetilde{X})= \widetilde{X} (\widetilde X - I)(\widetilde X + I)^{-3} \in \CC^{n \times n}$. Let
$Z(t) = Z(\widetilde{X}(t))=\widetilde{X}(t) (\widetilde X(t) - I)(\widetilde X(t) + I)^{-3}$, for $t \geq 0$, where $X(t) \in K$ is the solution of the closed-loop system (\ref{tesys},\ref{feedback}) with initial condition $\widetilde{X}(0) \in K-\{ I \}$. From the same arguments of the end of the proof of Theorem~\ref{tMain} (see
the relationship between $\widehat Z$ and $\widetilde X$), one concludes that  $Z_0 = Z(0) \neq 0$ and $Z_0 / \| Z_0 \| \in \mathbf{G}$. By uniqueness of solutions, $Z(0) \neq 0$ implies that $Z(t) \neq 0$ for $t \in [0, T]$.

\item[(d)] It will be shown by contradiction that there exists $N_{\mathbf{a}} > 0$ such that, for all $\widetilde{X}(0) \in K-\{I\}$, one has
\begin{equation}\label{timeineq}
\int_{0}^{T}
\sum_{k=1}^{m} f_k \mbox{Tr}^2(Z(t) \widetilde{S}_k(t)) \, dt  \geq N_{\mathbf{a}} \| Z_0 \|^2.
\end{equation}
Assume that this is not the case. Hence, for every $n \in \NN$ with $n > 0$, there exists $\widetilde{X}_n(0) \in K-\{ I \}$ such that
\begin{equation}\label{contradiction}
\int_{0}^{T}
\sum_{k=1}^{m} f_k \mbox{Tr}^2(Z_n(t) \widetilde{S}_k(t)) \, dt =  \int_{0}^{T}
\sum_{k=1}^{m} \left(\frac{\widetilde{u}_{k,n} (t)}{\sqrt{f_k}}\right)^2  \, dt <  \dfrac{1}{n} \| Z_{0n} \|^2,
\end{equation}
where $Z_{0n}=Z_n(0)$ and $\widetilde{u}_{k,n} (t) = \widetilde{u}_{k,n}(\widetilde{X}_n(t),t)$ is as in \eqref{feedback}.
Fix $n \in \NN$ with $n > 0$. The Cauchy-Schwartz inequality provides that (see \refeq{CS2}
of Appendix \refeq{D}), for $t \in [0,T]$,
\begin{equation}\label{utilde}
    \int_{0}^{t}  \left|\frac{\widetilde{u}_{k,n} (s)}{\sqrt{f_k}}\right|  \, ds \leq \int_{0}^{T}  \left|\frac{\widetilde{u}_{k,n} (s)}{\sqrt{f_k}}\right|  \, ds \leq  \sqrt{T} \sqrt{\int_{0}^{T} \left(\frac{\widetilde{u}_{k,n} (t)}{\sqrt{f_k}}\right)^2 \, dt}  < \sqrt{T/n} \| Z_{0n} \|.
\end{equation}
Now, using the fact that $X_n(t) \in K$, standard computations show that
\[
    \dot{Z}_n(t) = \sum_{k=1} R_k(X_n(t)) \widetilde{u}_{k,n}(t),
\]
where $R_k(t)=R_k(X_n(t))$ is continuous and uniformly bounded on $[0,T]$ by some $D_k > 0$. Thus, for every $t \in [0, T]$,
\begin{equation}
\label{eZn}
    \left\| \dfrac{Z_n(t) - Z_{0n}}{\|Z_{0n} \|} \right\| = \left\| \int_{0}^t \dfrac{\dot{Z}_n(s)}{\|Z_{n0} \|} \, ds \right\| \leq \sum_{k=1}^m D_k \int_{0}^{t} \frac{| \widetilde{u}_{k,n}(s) |}{\| Z_{0n} \|} \, ds.
\end{equation}
From \refeq{utilde} and \refeq{eZn}, it follows that
 \[
 \left\| \dfrac{Z_n(t) - Z_{0n}}{\|Z_{0n} \|} \right\|
\leq f D \sqrt{T/n},
\]
where $D=D_1 + \dots + D_m > 0$ and $f=\max\{\sqrt{f_1}, \dots, \sqrt{f_m}\} > 0$.
As the sequence $Z_{0n}/\|Z_{0n}\|$ belongs to the compact set $\mathbf{G}$, there exists a convergent subsequence. For simplicity, denote such subsequence by
$Z_{0n}/\|Z_{0n}\|$ and let $Z^* \in \mathbf{G}$ be its limit. It follows that $Z_n(t)/\|Z_{0n}\|$ uniformly converges to $Z^*$ on the interval $[0,T]$ as $n \to \infty$. Consequently, \eqref{constantineq} gives that
\[
    \lim_{n \to \infty} \int_{0}^{T} \sum_{k=1}^{m} \mbox{Tr}^2 \left(\dfrac{Z_n(t)}{\|Z_{n0} \|}  \widetilde{S}_k(t) \right) \, dt =  \int_{0}^{T} \sum_{k=1}^{m} \mbox{Tr}^2(Z^* \widetilde{S}_k(t)) \, dt \geq P_{\mathbf a}^* > 0.
\]
However, \eqref{contradiction} and the fact that $ \mbox{Tr}^2 \left(\dfrac{Z_n(t)}{\|Z_{n0} \|}  \widetilde{S}_k(t) \right) =  \dfrac{1}{\|Z_{n0} \|^2} \mbox{Tr}^2 \left(Z_n(t) \widetilde{S}_k(t) \right)$
implies that the limit above is zero, which is a contradiction.

\item[(e)] Using \eqref{negdef}, \eqref{z2ineq} and \eqref{timeineq}, one obtains
\begin{align*}
\mathcal V(\widetilde X((j+1)T)) -  \mathcal V(\widetilde X(jT)) & = \int_{jT}^{(j+1) T}
- 4 \sum_{k=1}^{m} f_k \mbox{Tr}^2({Z(t)} \widetilde{S}_k(t)) \, dt \\
& \leq
- 4 N_{\mathbf{a}} \| Z(jT) \|^2 \leq - 4 c_3 N_{\mathbf{a}} \| \widetilde{X}(jT) - I \|^2.
\end{align*}

\item[(f)] The results above considered a fixed $X_{goal}$ and the associated compact set $K$. Now, since $\mathcal{V}(I) = 0$, and by Theorem~\ref{t1} one has $\lim_{t \to \infty} \widetilde{X}(t) = I$ for any given initial condition $\widetilde{X}(0) = \overline{X}_{goal} \in \mathcal{W}$ of the closed-loop system, it is clear that there always exists $\overline{t} \geq 0$ such that $\widetilde{X}(\overline{t}) \in K$.

\end{itemize}
\end{proof}

The proof of Theorem~\ref{expconv} for  the stochastic strategy has a similar structure, and is now presented.

\begin{proof}
In this proof one will denote by $\widetilde{X}^{\mathbf a}(t)$ the solution of the closed loop system
(\ref{tesys},\ref{feedback}) with initial condition $\widetilde{X}(0)= X_0$, where $\mathbf{a}$ is
the set of amplitudes $a_{k,\ell}, k=1, \ldots, M, \ell=1, \ldots, M$, defining the reference control \refeq{refcon}.
Then $Z^{\mathbf a}(t)$ is a skew-Hermitean matrix given by  $Z^{\mathbf a}(t) = \widetilde{X}^{\mathbf a}(t) (\widetilde{X}^{\mathbf a}(t) -I) (\widetilde{X}^{\mathbf a}(t)+ I)^{-3}$. Note that if $X_0 \in K$, where $K$ is the positively invariant compact set defined by \refeq{eK},
then $Z^{\mathbf a}(t)$ is uniformly bounded on $\mathbf{A} \times [0,\infty)$. Furthermore, as in the deterministic strategy, by uniqueness of solutions, if $\widetilde{X}_0 = \widetilde{X}(0) \neq I$, it is clear that $\widetilde{X}^{\mathbf a}(t) \neq I$ and $Z^{\mathbf a}(t) \neq 0$ for all $t \geq 0$. Let $Z_0 = Z^\mathbf{a}(0)$.

According to  \cite[Theorem 2(8.8c), p.\ 197]{Kushner71} and from properties \refeq{ec1c2}
 and \refeq{z2ineq} of Lemma \ref{lExp},
it suffices to show that there exists  $M>0$ such that $Q(\widetilde{X}_0) \geq M \| Z_0 \|^2$, for all $\widetilde{X}_0 \in K$, where $Q$ is defined by \refeq{eDefQ}. Assume the contrary. Then,
from \refeq{feedback} and \refeq{negdef}, one has that for all $n \in \NN$ with $n>0$, there exists $\widetilde{X}_{0n} \in K - \{I\}$ such that
\begin{equation}
 \label{eContrS}
  \int_{\mathbf{A}} \left[\int_{0}^{T}  \sum_{k=1}^{m} f_k \mbox{Tr}^2
\left( Z^{\mathbf a}_n(t) \widetilde{S}^{\mathbf a}
_k(t)\right) dt \right] d \mathbf{a}  = \int_{\mathbf{A}} \int_{0}^{T} \left[ \sum_{k=1}^{m} \left(\frac{\widetilde{u}_{k,n}^\mathbf{a} (t)}{\sqrt{f_k}} \right)^2 dt \right] d \mathbf{a}<   \frac{1}{n} \|Z_{0n}\|^2.
\end{equation}
Using \refeq{eContrS}, the Cauchy-Scharwz inequality (see \refeq{CS2} in Appendix \ref{D}),
and the fact that $\mbox{vol}\left( \mathbf {A} \times [0, T] \right)=
\mbox{vol}\left( \mathbf {A} \right) T$ (Lebesgue measure), one obtains:
 \begin{equation}
 \label{eutildea}
\int_{\mathbf{A}} \int_{0}^{t} \left| \left(\frac{\widetilde{u}_{k,n}^\mathbf{a} (t)}{\sqrt{f_k}} \right)\right|  dt \, d \mathbf{a} \leq  \int_{\mathbf{A}} \int_{0}^{T} \left| \left(\frac{\widetilde{u}_{k,n}^\mathbf{a} (t)}{\sqrt{f_k}} \right)\right|  dt \, d \mathbf{a} \leq \sqrt{\mbox{vol}(\mathbf A)T/n}  \| Z_{0n}\|, \quad t \in [0,T].
\end{equation}
Now, by the same reasoning that was used to obtain \refeq{eZn}, one may write
\begin{equation}
\label{eZna}
   \int_{\mathbf A} \left\| \dfrac{Z^{\mathbf a}_n(t) - Z_{0n}}{\|Z_{0n} \|} \right\| d{\mathbf a} = \left\| \int_{\mathbf A} \int_{0}^t \dfrac{\dot{Z}_n^{\mathbf a}(s)}{\|Z_{n0} \|} \, ds \, d{\mathbf a}\right\| \leq  \int_{\mathbf A} \sum_{k=1}^m D_k \int_{0}^{t} \frac{\left| \widetilde{u}_{k,n}^{\mathbf a}(s) \right|}{\| Z_{0n} \|} \, ds \, d {\mathbf a}, \quad t \in [0,T],
\end{equation}
where $D_k > 0$, and so by \refeq{eutildea} and \refeq{eZna} it follows that
\begin{equation*}
   \int_{\mathbf A} \left\| \dfrac{Z^{\mathbf a}_n(t) - Z_{0n}}{\|Z_{0n} \|} \right\| d{\mathbf a} \leq
  f D \sqrt{ \mbox{vol} (A) T/n}, \quad t \in [0, T],
\end{equation*}
where $D=D_1+ D_2+ \ldots + D_m > 0$ and $f = \max \{ \sqrt{f_1}, \ldots, \sqrt{f_m}\}  > 0$. From this last equation, one gets
\begin{equation}
   \label{L1}
   \int_{\mathbf A} \int_0^T \left\| \dfrac{Z^{\mathbf a}_n(t) - Z_{0n}}{\|Z_{0n} \|} \right\|  dt \, d{\mathbf a} \leq
  T f D \sqrt{ \mbox{vol} (A) T / n}.
\end{equation}

Up to a convenient subsequence, one may assume that the sequence $\frac{Z_{0n}}{\|Z_{0n} \|}$ converges
to some $Z^* \in \mathbf G$. Then \refeq{L1} implies that  $\frac{Z^{\mathbf a}_{n}(t)}{\|Z_{0n} \|}$
converges to $Z^*$ in the $L^1$ sense:
\begin{equation}
   \label{L1L2}
   \lim_{n \to \infty} \int_{\mathbf A} \int_0^T \left\| \dfrac{Z^{\mathbf a}_n(t)}{\|Z_{0n}  \|} - Z^* \right\| dt \, d{\mathbf a} = 0.
\end{equation}
Note that the sequence $Z^{\mathbf a}_{n}(t)$ is uniformly bounded on $\mathbf{A} \times [0, T]$ and that \refeq{eContrS} implies that
\begin{equation}
 \label{eContrS2}
  \int_{\mathbf{A}} \left[\int_{0}^{T}  \sum_{k=1}^{m} f_k \mbox{Tr}^2
\left( \frac{Z^{\mathbf a}_n(t)}{\|Z_{0n} \|} \widetilde{S}
_k^{\mathbf a}(t)\right) dt \right] d \mathbf{a}  <   \frac{1}{n}.
\end{equation}
Furthermore, note that $\mbox{Tr}^2
\left( \frac{Z^{\mathbf a}_n(t)}{\|Z_{0n} \|} \widetilde{S}_k(t)\right)$ is a sum of products
of the form
$z_{ij} \, z_{kl} \, \widetilde{s}_{pq} \, \widetilde{s}_{rs}$,
where  $z_{ij}$ denotes an element of $\dfrac{Z^{\mathbf a}_n(t)}{\|Z_{0n} \|}$ and
$\widetilde{s}_{pq}$ denotes an element of $\widetilde{S}_k^{\mathbf a}(t)$. From Lemma
\ref{D1} of Appendix \ref{D}, it follows that, by taking the limit
 $n \rightarrow \infty$, one may replace $\dfrac{Z^{\mathbf a}_n(t)}{\|Z_{0n} \|}$
by its limit $Z^* \in \mathbf G$ in the integral \refeq{eContrS2}:
\begin{equation*}
  \lim_{n\rightarrow\infty} \int_{\mathbf{A}}\int_{0}^{T}  \sum_{k=1}^{m} f_k \mbox{Tr}^2
\left( \frac{Z^{\mathbf a}_n(t)}{\|Z_{0n} \|} \widetilde{S}
_k^{\mathbf a}(t)\right) dt  \, d \mathbf{a}  = \int_{\mathbf{A}} \int_{0}^{T}  \sum_{k=1}^{m} f_k \mbox{Tr}^2
\left( Z^* \widetilde{S}
_k^{\mathbf a}(t)\right) dt  \, d \mathbf{a}.
\end{equation*}

Now note that, for a fixed $Z^* \in \mathbf G$, then \refeq{constantineq} holds for almost all
$\mathbf a \in \mathbf A$ (cf.\ Theorem~\ref{t1}). The continuous dependence of the $T$-periodic reference trajectory
$\overline{X} (t)$  with respect to the set of
parameters ${\mathbf a}$ implies that the $T$-periodic map $\widetilde{S}^{\mathbf a}_k(t)$
also depends continuously on $\mathbf a$, and hence\footnote{Note that in this
argument $Z^* \in \mathbf G$ is a fixed matrix.}
\[
\lim_{n\rightarrow\infty} \int_{\mathbf{A}} \int_{0}^{T}  \sum_{k=1}^{m} f_k \mbox{Tr}^2
\left( \frac{Z^{\mathbf a}_n(t)}{\|Z_{0n} \|} \widetilde{S}_k^{\mathbf a}(t)\right) dt \;  d \mathbf{a} > 0.
\]
However, \refeq{eContrS2} implies that the same limit is zero, which is a contradiction.
\end{proof}

\begin{remark}\label{worstdirec}
 The last proof and \refeq{z2ineq} imply that
 $Q({\widetilde X_0}) \geq M \| {\widetilde X_0} - I\|^2$. Note that
 this does not explain why the convergence of the stochastic method is faster.
  The authors believe that this is due to the following fact. Consider a realization
  of the stochastic method such that \refeq{Ma} holds in each step
  of both deterministic and stochastic methods.
  Note that $M_{\mathbf a}$ regards the worst direction, that is, for some  ${\widetilde X}_0={\widetilde X}_0^{\mathbf a}$ one has that $V({\widetilde X}_f) - V({\widetilde X}_0)$ is of order $M_{\mathbf a} \| {\widetilde X}_0 - I\|^2$. It seems that the ``worst direction'' ${\widetilde X}_{0}^{\mathbf a}$ is very
sensible to $\mathbf a$, assuring that the next aleatory steps will provide a compensation of the
speed just by varying the worst direction, and then providing a better speed in an average process.
\end{remark}

\section{Concluding Remarks}
\label{sConclusions}

In this work one has proposed a constructive solution of the
$T$-sampling stabilization problem for quantum systems on $\mbox{U}(n)$. It is easy to show\footnote{Note that
$\beta(X) = \det (X + I)$ is a polynomial function in the entries $X_{ij}$ of $X=(X_{ij})$, and as it is not identically zero, the set $\mathcal R
\subset \CC^{n\times n}$ of its roots is closed and it has zero Lebesque
measure zero (\cite{polynomial}).} that the complement $\mathcal{W}^c$ of the  set $\mathcal W$ in \eqref{eW} is closed with Lebesgue measure zero, and $\mathcal{W}$ is dense in $\mbox{U}(n)$.
It was also shown that, if one accepts a global phase change on $X_{goal}$ of the form $X_{goal} \in \mbox{U}(n) \mapsto \exp(\iota \phi)X_{goal} \in \mathcal{W} \subset \mbox{U}(n)$, which is transparent for quantum systems,
one may always obtain an equivalent $X_{goal}$ inside $\mathcal W$.
From
this perspective, the two steps procedure of Theorem \ref{tSteer}
for $T$-sampling stabilization may be completely avoided.
 However, note
from \refeq{feedback} that the feedbacks $\widetilde u_k$ tends
to infinity when the initial condition $\widetilde X(0)$ tends to
$\mathcal{W}^c$. For initial conditions that are close to
$\mathcal{W}^c$, simulations have shown a better compromise of the
convergence speed versus the maximum norm of the input when one
chooses the two-step procedure instead of the single one.

The approach of this paper is based on previous results of
\cite{SilPerRou09,SilPerRou12} which provided a solution of the $T$-sampling stabilization problem for controllable systems \eqref{cqs} on $\mbox{SU}(n)$ in a certain number of steps that may
grow with $n$. In the present work, a new Lyapunov
function $\mathcal{V}$ is defined. This new choice assures a complete and global solution of
the problem when the system is controllable, since $\mathcal{V}$ essentially decreases without singularities and with no nontrivial LaSalle's invariants.
The results of \cite{SilPerRou09}
are based on the existence of a special kind of $T$-periodic reference
trajectory which is generated by special inputs,
called here Coron reference controls\footnote{Their
existence is the heart of Coron's Return Method.}, in a way
that the linearized system is controllable along such trajectory.  In
\cite{SilPerRou09} it is not shown how to construct a trajectory with such
properties, and in \cite{SilPerRou12} it is indeed
constructed when the system obeys the $p$-controllability
condition. This paper relaxed such condition by assuming only that the system is controllable.
Another important contribution of this work was to show that,
for a number of frequencies $M > 0$ big enough, the inputs of the form  $\overline{u}_k(t)
= \sum_{\ell=1}^{M} a_{k,\ell} \sin ( 2 \pi t /T )$ generate Coron
reference controls with probability one with respect to a random
choice of the real coefficients $a_{k,\ell}$. This result implies that
the control laws required in the present work and in \cite{SilPerRou09} are completely constructible. Furthermore, one has established that the convergence
of the $T$-sampling stabilization problem is exponential for both methods, deterministic an stochastic.
It is important to point out that the speed of exponential convergence may be controlled if one chooses $\overline{T}=T/c$,
$\overline{f}_k = c f_k$, $\overline{\mathbf{a}}=c\mathbf{a}$, where $c > 0$. This is immediate from Remark~\ref{atp}.

The methods presented here could be adapted to controllable quantum models \eqref{cqs} that evolve on the special unitary group $\mbox{SU}(n)$ as well, but this will be the subject of a future work.

\section*{Acknowledgements}

The first author was fully supported by CAPES and FUNPESQUISA/UFSC.
The second author was partially supported by CNPq, FAPESP and
USP-COFECUB.
The third author was partially supported by ``Agence
Nationale de la Recherche'' (ANR), Projet Blanc EMAQS number ANR-2011-BS01-017-01,  and USP-COFECUB.

\appendix


\section{Compactness of $K$}
\label{aCompact}

\begin{proposition}
 The set  $K$ in \refeq{eK} is compact in $\mathcal{W}$.
\end{proposition}

\begin{proof}
As $\mbox{U}(n)$ is compact, it suffices to show that $K$ is
closed in $\mbox{U}(n)$. For this, assume that $W_k \in K$ is a sequence with
$\lim_{k \rightarrow \infty} W_k = \overline W \in \mbox{U}(n)$. If
$\overline W \in \mathcal W$, then the continuity of $\mathcal V$ in
$\mathcal W$ gives
$ \lim_{k \rightarrow \infty} \mathcal V(X_k) = \mathcal V ( \lim_{k \rightarrow
 \infty}X_k) = \mathcal V ( \overline W)$.
 As $\mathcal V(W_k) \leq
 \mathcal V(X_{goal})$, then one must have  $ \mathcal V ( \overline
 W)  \leq \mathcal V(X_{goal})$ and $\overline W \in K$. To conclude
 the proof, it suffices to show that $\overline W$ must be in $\mathcal W$.
For that purpose, suppose that at least one eigenvector of $\overline W$ is
 equal to $-1$. Recall that $\mathcal W$ is defined in \refeq{eW}, i.e.\ $W$ is the set of
all matrices $W \in \mbox{U}(n)$ that do not have any eigenvalue
equal to $-1$. Furthermore, one may always write $W \in \mbox{U}(n)$ as $W = U^\dag D U$, where $U$ unitary, $D
=\diag(\exp(\iota \theta_1), \ldots, \exp(\iota \theta_n))$ is a diagonal matrix, and $\exp(\iota \theta_j), j=1, \ldots, n$, are the eigenvalues of $W$. Using \refeq{lyap} and the invariance of the trace, it follows that
\begin{equation}\label{lyaptan}
    \mathcal{V}(W) = \sum_{j=1}^n \frac{(\exp(\iota \theta_j) -1)^2}{(\exp(\iota \theta_j)+1)^2}=\sum_{j=1}^n \{ \tan(\theta_j/2) \}^2, \quad \mbox{for } W \in \mathcal{W}.
\end{equation}
In particular, from the continuous dependence of the spectrum
$\sigma( W)$ with respect to $W$, if $W_k \rightarrow \overline W$
with $W_k \in K$, then one must have $\mathcal V ( W_k ) \rightarrow
+ \infty$, because at least one eigenvalue of $\overline W$ is
equal to $-1$, This contradicts the fact that $\mathcal V ( W_k ) \leq
\mathcal V(X_{goal})$ (see \refeq{eK}).
\end{proof}

\section{LaSalle's Invariance Theorem for Periodic Systems}\label{aLasalle}

LaSalle's invariance theorem also holds for periodic systems. It can
be regarded as a generalization of previous results of E.A.
Barbashin and N.N. Krasovski \cite[Theorem 1.3, p. 50]{RouHabLal}.
The following result is a complex version of
\cite[Theorem 3, p.\ 10]{LaSalle}.

\begin{theorem}
\label{tLaSalle}  Let $S \subset \CC^n$ be an open set and $T > 0$.
Consider that $f$:~$S \times \RR  \rightarrow \CC^n$ is continuously differentiable.
Assume that $f$
is $T$-periodic in $t$, that is $f(x,t+T)= f(x,t)$, for all $x \in S$ and $t \in
\RR$. Consider the system
\begin{equation}
\begin{array}{rcl}
 \label{sLaSalle}
 \dot x(t) & = & f(x(t),t), \smallskip \\
 x(t_0) & = & x_0.
 \end{array}
\end{equation}
Let $K \subset S$ be a compact set and suppose that $K$ is a
positively invariant set  of the dynamics \refeq{sLaSalle}. Let
$\mathcal V$:~$S \times \RR \rightarrow \RR$ be a $T$-periodic and continuously differentiable
function such that $\dot {\mathcal V} (x,t) = \frac{\partial \mathcal V(x,t)}{\partial x} f(x,t) + \frac{\partial
\mathcal V(x,t)}{\partial t}  \leq 0$, for all $x \in S$ and $t \in \RR$.
Let
 \[
E = \left\{ (\widehat x,t) \in  K \times [t_0, \infty) \; | \; \dot
{\mathcal V}(\widehat x,t) = 0\right\}.
 \]
Let $\mathcal M$ be the union of all the trajectories  $(x(t),t)$, $t
\geq t_0$, contained in $E$. Then, every solution $x(t)$ with initial
condition $x(t_0)= x_0 \in K$ asymptotically converges to $\mathcal M$, that is $\lim_{t
\rightarrow \infty}\mbox{dist} (x(t), \mathcal{M}) = 0$.
\end{theorem}

\section{Proof of Theorem \ref{tProb1}}
\label{aProb1}

Assume that there exist a Coron reference control
$\widehat{u}_k(t)$, $k=1,\dots,m$. Let $u_k(t)$, $k=1,\dots,m$, be
any smooth $T$-periodic controls in \eqref{cqs} with initial
condition $I$ at $t=0$. Consider that such controls obey
\eqref{sinus}. In particular, each $u_k(t)$ is an odd function.

Along this proof, one lets
$\left(\frac {I}{J} \right)$ stand for the integer division of $I$ by $J$.
The proof follows easily from the arguments below:

(i) Using \eqref{defCkj}, one gets
\begin{align*}
 & C_k^0(t)=S_k,
 \; C_k^1(t)=\sum_{\ell_1} u_{\ell_1}(t) [S_k,S_{\ell_1}], \\
 & C_k^2(t)=\sum_{\ell_1} \dot{u}_{k}(t) [S_k,S_{\ell_1}] + \sum_{\ell_1,
\ell_2}
u_{\ell_1}(t) u_{\ell_2}(t) [[S_k,S_{\ell_1}],S_{\ell_2}].
\end{align*}
It is then easy to show by induction that
\begin{equation}
 \label{CI}
  C_k^{j}(0) = \sum_{I \in \Lambda_{k,j}} \beta_I(U_{j-1}) \mathcal{H}_I,
\end{equation}
where $\beta_I$ is a monomial in the variables
$ U_{j-1} = \{u_k^{(p)}(0) : k=1, \ldots, m, \; p =0,
  \ldots, j-1 \}$,  $\Lambda_{k,j} = \{1, \ldots , n_{kj} \} \subset \NN$ and
$\mathcal{H}_I \in \mathfrak{u}(n)$. In order to be consistent,
both $U_{-1}$ and $U_{0}$ will be taken as empty sets. Since each
$u_k(t)$ is an odd function, it is easy to verify that\footnote{The
derivative of an even function is an odd function and
\emph{vice-versa}. In particular, the derivative at zero of an even
function is zero.} $u_k^{(j)}(0) = 0$ for every even $j \geq 0$. In
particular, for any smooth $T$-periodic inputs obeying
\eqref{sinus}, for instance the Coron reference controls
$\widehat{u}_k$ and the $\overline{u}_k$'s given in \refeq{refcon},
one may restrict $U_{j-1}$ to the variables
\begin{equation}
 \label{Uj}
  U_{j-1} = \left\{u_k^{(p)}(0) : k=1, \ldots, m  \; \mbox{and} \; p = 2 i - 1,
  \; \mbox{for} \; i=1, \ldots , \left( \tfrac{j}{2} \right) \right\}.
\end{equation}

(ii) Let $\{ G_h : h =1, \ldots , d\}$ be a given basis of
$\mathfrak{u}(n)$ as a real vector space. For each $i=1, \dots, d$, let
$\pi_i$:~$\mathfrak{u}(n)
\rightarrow \RR$ be the linear map defined by $\pi_i( \sum_{h=1}^{d}
\alpha_h G_h) = \alpha_i$. Define the real $d$-square matrix $P=(P_{ih})$
elementwise by
\[
   P_{ih} = \pi_i(C_{k_h}^{j_h}(0)),
\]
where the indices $j_h, k_h$ are the ones given in \refeq{eGenerates} for the
Coron reference controls $\widehat{u}_k(t)$. Since the $\pi_i$ are linear,
from \refeq{CI} one gets
\[
  P_{ih} =  \sum_{I \in \Lambda_{k_h,j_h}} \beta_I(U_{j_h-1})
  \pi_i(\mathcal{H}_I).
\]
Note that  $\pi_i(\mathcal{H}_I) \in \RR$, for $I \in
\Lambda_{k_h, j_h}$. In particular the entries $P_{ih}$ are polynomial
functions in the variables $U_{j_h-1}$ defined by \refeq{Uj}. Therefore, $\det
P$ is a polynomial function in the variables
\[
U = \{u_k^{(1)}(0), u_k^{(3)}{(0)}, \ldots,  u_k^{(s)}{(0)} : k =1,
\ldots , m, \; s=2 M - 1 \},
\]
where $M=\left( \frac{J}{2} \right)$. Note that $\det P(U)$ is
nonzero when computed for the vector $U$ determined by the values
$u_k^{(s)}(0)=\widehat{u}_k^{(s)}(0)$ corresponding to the
Coron reference controls $\widehat{u}_k(t)$.

(iii) Recall that $M=\left(\frac{J}{2}\right)$ and $s=2 M -1 = 2 (\frac{J}{2}) -
1$. Let $\mathcal{U}_k$ and $\widetilde{a}_k$ be the $M$-dimensional
column vectors in $\RR^M$ defined respectively as
\[
 \mathcal{U}_k = (u_k^{(1)}(0), u_k^{(3)}{(0)}, \ldots,
 u_k^{(s)}{(0)}), \qquad
 \widetilde{a}_k = (a_{k1}, a_{k2}, \ldots, a_{kM}).
\]
Now, take $u_k(t)=\overline{u}_k(t)$ as the ones given by
\refeq{refcon}. Then, it is easy to show that
\[
 \mathcal{U}_k = V_k \widetilde{a}_k,
\]
where $V_k = V$ is the real $M\times M$ matrix given by
\[
V = \left[
\begin{array}{cccc}
\omega     &   2   \omega    &    \ldots &     M \omega \\
-\omega^3   &  -2^3 \omega^3   &    \ldots  &    -M^3 \omega^3 \\
\omega^5   &  2^5 \omega^5   &    \ldots  &    M^5 \omega^5 \\
\vdots     &  \vdots         &    \vdots  &    \vdots \\
(-1)^{\left( \frac{s}{2} \right)}\omega^s   &  (-1)^{\left(
\frac{s}{2} \right)} 2^s \omega^s   & \ldots & (-1)^{\left(
\frac{s}{2} \right)} M^s \omega^s
\end{array}
 \right],
\]
where $\omega = 2\pi /T > 0$.
By dividing the first column by $\omega$, the second
column by $-2 \omega$ and so on, one gets the Vandermonde matrix
\[
V_1= \left[
\begin{array}{cccc}
1              &   1                   &    \ldots &     1 \\
\omega^2       &  2^2 \omega^2         &    \ldots &  M^2 \omega^2 \\
\omega^4       &  2^4 \omega^4         &    \ldots & M^4 \omega^4\\
\vdots         &  \vdots               &    \vdots & \vdots \\
\omega^{s-1} &  2^{s-1} \omega^{s-1} & \ldots & M^{s-1} \omega^{s-1}
\end{array}
\right].
\]
In particular, $\det V \neq 0$.

(iv) Consider the $mM$-column vectors $U = (\mathcal{U}_k : k
=1, \ldots, m)$ and $\widetilde{a} = ({\widetilde a}_k : k =1, \ldots,
m)$. Then,
\[
 U = V_B \widetilde{a},
\]
where $V_B$ is a block diagonal matrix with diagonal entries
given by $V$. As $V$ is nonsingular,  it is possible to choose the
$a_{k,\ell}$'s in \refeq{refcon} in a way that the vector $U$, when computed for
$u_k(t)=\overline{u}_k(t)$ given in \refeq{refcon}, coincides with the
vector $U$ computed for the Coron reference control $\widehat{u}_k(t)$. Hence
$\det P( \mathcal{V_B} {\widetilde a}) \neq 0$
for such choice of the $a_{k\ell}$'s, which will be denoted as
${\overline{ a}}_{k\ell}$, $k=1, \ldots, m$, $\ell=1, \ldots, M$.

(v)  Recall that the set of zeros  of a nonzero polynomial function
$\mathcal P$:~$\RR^{m M} \rightarrow \RR$ has zero Lebesgue measure,
with an open and dense complement\footnote{For an elementary proof
of this fact, see e.g. \cite{polynomial}. This result holds in fact for
analytic functions in a much more general situation (see \cite{analytic}).}.
Now, regarding $\det P(V_B {\widetilde a})$
as a polynomial function in the variables $a_{k\ell}$ of ${\widetilde a}$,
since this function is nonzero when computed for the
${\overline{a}}_{k\ell}$'s constructed above, one has proved the theorem.
Note that if one takes $M > (\frac{J}{2})$, the same proof applies. The only difference is that $V_B$ is no longer a square matrix, but it has full row rank.

\section{Cauchy-Schwarz inequality for Lebesgue measure}
\label{D}

In this appendix one recalls some properties of measurable functions. Let $S \subset \RR^n$ be
a Lebesque measurable subset with finite measure $\mbox{vol} (S)$. Let $f$:~$S \rightarrow \RR$ and $g$:~$S \rightarrow \RR$ be two measurable
functions. Then the well-known Cauchy-Schwarz inequality says that
\begin{equation}
 \label{CS}
 \int_S | f(x) g(x) | dx \leq \sqrt{\int_S f^2(x) dx} \sqrt{\int_S g^2(x) dx}.
\end{equation}
Talking $f \equiv 1$  one gets
\begin{equation}
 \label{CS2}
 \int_S  |g(x)|  dx \leq  \sqrt{\mbox{vol} (S) \int_S g^2(x) dx}.
\end{equation}
Assume that $h$:~$S \rightarrow \RR$ is measurable with $|h(x)| \leq M$ uniformly on $S$. Since $h^2(x) = | h (x) |^2 \leq M | h(x) |$ uniformly on $S$, one gets
 \begin{equation}
 \label{CS3}
 \int_S  h^2(x)  dx \leq  M \int_S |h(x)| dx.
\end{equation}
The following lemma is useful:
\begin{lemma}
 \label{D1}
Assume that $f_n$:~$S \rightarrow \RR$ and  $g_n$:~$S \rightarrow \RR$ is a uniformly bounded sequence of measurable functions that converge in the $L^1$ sense to
uniformly bounded measurable functions, that is, there exist
functions $\phi$:~$S \rightarrow \RR$ and $\gamma$:~$S \rightarrow \RR$ such that
\begin{equation*}
 \lim_{n\rightarrow\infty} \int_S | f_n(x) - \phi(x) | \, dx = \lim_{n\rightarrow\infty} \int_S | g_n(x) - \gamma(x) | \, dx  =  0,
 \end{equation*}
where $\phi$:~$S \rightarrow \RR$ and $\gamma$:~$S \rightarrow \RR$  are
measurable and uniformly bounded. Then, for every measurable and uniformly
bounded function $v$:~$S \rightarrow \RR$, the following properties hold:
\begin{eqnarray*}
\label{aP} \lim_{n\rightarrow\infty} \int_S  f_n(x) v(x) dx & = & \int_S  \phi(x) v(x) \, dx, \\
\label{bP} \lim_{n\rightarrow\infty} \int_S  f_n(x) g_n(x) v(x) \, dx & = & \int_S  \phi(x) \gamma(x) v(x) \, dx.
\end{eqnarray*}
\end{lemma}
\begin{proof} This is a standard result in functional analysis.
For simplicity, assume that $M>0$ is a uniform bound for all the functions $\phi, \gamma, v, f_n, g_n$, $n \in \NN$.
To show the first claim, it suffices to see that (one abuses notation for simplicity)
$| \int f_n v - \phi v | \leq \int |v| |f_n - \phi|  \leq M \int |f_n- \phi|$.
To show the second claim,
one shows first that $\int_S |f_n g_n - \phi \gamma | dx \rightarrow 0$.
For this, note that
\begin{eqnarray*}
|f_n g_n - \phi \gamma | & = & |(f_n - \phi) (g_n - \gamma) +  \phi(g_n- \gamma) + \gamma (f_n - \phi) | \\
& \leq & |(f_n - \phi) (g_n - \gamma)| + |\phi| | g_n- \gamma| + |\gamma| | f_n - \phi|.
 \end{eqnarray*}
Furthermore, note that \refeq{CS3} and the $L_1$ convergence assumption imply $L_2$ convergence, that is, $\int_S |f_n  - \phi |^2 dx \rightarrow 0$ and $\int_S |g_n  - \gamma |^2 dx \rightarrow 0$.
Then, the fact that $\int_S |f_n g_n - \phi \gamma | dx \rightarrow 0$ follows easily from
 Cauchy-Scharwz inequality \refeq{CS} and from the fact
that the functions $\phi, \gamma$ are uniformly bounded. Since $v$ is uniformly bounded,
 the second claim follows
easily from the first one.
\end{proof}

\section{Stochastic Lyapunov Stability
Result}\label{stochasticinvariance}

\begin{theorem}\label{stochastic}
\cite[Theorem 1, p.\ 195]{Kushner71} Let $\Omega$ be a probability
space and let $\mathcal{W}$ be a measurable space.
Consider that $X_j$:~$\Omega \rightarrow \mathcal{W}$, $j \in \NN$,
is a Markov chain with respect to the natural filtration. Let
$Q$:~$\mathcal{W} \rightarrow \RR$ and
$\mathcal{V}$:~$\mathcal{W} \rightarrow \RR$ be measurable
non-negative functions with $\mathcal{V}(X_j)$ integrable for all $j
\in \NN$. If
\[
  \mathbb{E}\left(\mathcal{V}(X_{j+1}) / X_{j}\right) - \mathcal{V}(X_j) = - Q(X_j), \quad
\mbox{for } j \in \NN,
\]
then $\lim_{j \to \infty} Q(X_j) = 0$ almost surely.
\end{theorem}

\def\cprime{$'$}

\end{document}